\documentclass[12pt]{article} 
\usepackage[utf8]{inputenc}
\usepackage[english]{babel}
\usepackage{amsfonts}
\usepackage{amstext}
\usepackage{graphicx}
\usepackage{amsmath}
\usepackage{amssymb}
\usepackage{euscript}
\usepackage{amsthm}
\usepackage{enumerate}
\numberwithin{equation}{section} 

\renewcommand{\Re}{\mathop{\rm Re}\nolimits}
\renewcommand{\Im}{\mathop{\rm Im}\nolimits}

\graphicspath{{Graphics/}}
\newtheorem{theorem}{Theorem}
\newtheorem{proposition}{Proposition}
\newtheorem{remark}{Remark}

\begin{document}
\title{Defocusing nonlocal nonlinear Schr\"odinger equation with 
step-like boundary conditions: 
 long-time behavior for  shifted  initial data}
\author{Ya. Rybalko$^{\dag}$ and D. Shepelsky$^{\dag,\ddag}$\\
 \small\em {}$^\dag$ B.Verkin Institute for Low Temperature Physics and Engineering\\ \small\em {} of the National Academy of Sciences of Ukraine\\
 \small\em {}$^\ddag$ V.Karazin Kharkiv National University}


\maketitle

\begin{abstract}
The present paper deals with the long-time asymptotic analysis of the initial value problem for the integrable defocusing nonlocal nonlinear Schr\"odinger equation
$
iq_{t}(x,t)+q_{xx}(x,t)-2 q^{2}(x,t)\bar{q}(-x,t)=0
$
with a step-like initial data: 
$q(x,0)\to 0$ as $x\to -\infty$ and $q(x,0)\to A$ as $x\to +\infty$.
Since the equation is not translation invariant, the solution of this problem
is sensitive to  shifts of the initial data. 
We consider a family of problems, parametrized by $R>0$,
with the initial data that can be viewed as perturbations of 
 the ``shifted step function'' $q_{R,A}(x)$: 
$q_{R,A}(x)=0$ for $x<R$ and $q_{R,A}(x)=A$ for $x>R$, where $A>0$ and $R>0$ are arbitrary constants. 
We show that the asymptotics is qualitatively different in sectors of the 
$(x,t)$ plane, the number of which depends on the relationship between 
$A$ and $R$: for a fixed $A$, the bigger $R$, the larger number of sectors.
Moreover, the sectors can be collected into 2 alternate groups: 
in the sectors of the 
first group, the solution decays to 0 while in the sectors of the second group,
the solution approaches a constant (varying with the direction $x/t=const$).

\end{abstract}

\section{Introduction}

The nonlocal  nonlinear Schr\"odinger (NNLS) equation was introduced by M. Ablowitz and Z. Musslimani \cite{AMP} as a particular reduction of the Schr\"odinger system \cite{AKNS}
\begin{equation}
\label{schrsys}
\begin{cases}
iq_t(x,t)+q_{xx}(x,t)+2q^2(x,t)r(x,t)=0\\
-ir_t(x,t)+r_{xx}(x,t)+2q(x,t)r^2(x,t)=0
\end{cases}
\end{equation}
where $r(x,t)=\sigma\bar{q}(-x,t)$ with $\sigma=\pm1$.  Here and below, $\bar{q}$ denotes the complex conjugate of $q$. Recall that the reduction
 $r(x,t)=\sigma\bar{q}(x,t)$ in (\ref{schrsys}) leads to  the conventional (local) nonlinear Schr\"odinger (NLS) equation. The NNLS equation has attracted a considerable interest due to its distinctive physical and mathematical properties. Particularly, it is an integrable equation \cite{AMP, GS}, which can be viewed as a $PT$ symmetric \cite{BB} system, i.e., if $q(x,t)$ is its solution, so is $\bar{q}(-x,-t)$. Therefore, the NNLS equation is closely related to the theory of $PT$ symmetric systems,  which is a cutting edge field in modern physics, particularly, in optics and photonics, (see e.g. \cite{ZB, B, BKM, ZQCH, GA, KYZ} and references therein). 
Also, (\ref{1-a}) is a particular case of the so-called Alice-Bob integrable systems, which describe various physical phenomena occurring in two (or more)  different  places somehow linked to each other \cite{Lou, LH}: if this two different places are not neighboring, the corresponding model becomes nonlocal. 
Being 
 an integrable system, the NNLS equation possesses exact, soliton-like solutions,
which, however,  have  unusual properties: particularly, solitons can blow up in a finite time, and the equation supports, simultaneously, both dark and anti-dark soliton solutions  (see e.g. \cite{SMMC, Y, YY, CZ, GP, AMFL, AMN, MS}), which is in a sharp contrast with the conventional (local) nonlinear Schr\"odinger equation.

We consider the  Cauchy problem for the defocusing NNLS equation with the step-like initial data:
\begin{subequations}
\label{1}
\begin{align}
\label{1-a}
& iq_{t}(x,t)+q_{xx}(x,t)-2 q^{2}(x,t)\bar{q}(-x,t)=0,&& x\in\mathbb{R},\,t>0,\\
\label{1-b}
& q(x,0)=q_0(x),&& x\in\mathbb{R}
\end{align}
(which corresponds to $\sigma=-1$), where
\begin{equation}
\label{shifted-step-gen}
q_0(x)\to
\begin{cases}
0,\quad x\to -\infty,\\
A,\quad x\to \infty,
\end{cases}
\end{equation}
\end{subequations}
with some $A>0$.
We assume  that the solution $q(x,t)$ satisfies the  boundary conditions 
(consistent with the equation) for all 
$t\geq0$:
\begin{equation}
\label{2}
q(x,t)\to
\begin{cases}
0,\quad x\to -\infty,\\
A,\quad x\to \infty,
\end{cases}
\end{equation}
where the convergence is sufficiently fast.

The choice of the boundary values (\ref{2}) is inspired by the recent progress in 
studying problems with step-like (generally, asymmetric) boundary conditions for conventional (local) integrable equations, such as the Korteweg-de Vries (KdV) equation \cite{GP73, EG, AE1}, the modified KdV equation \cite{KM}, the Toda lattice \cite{VDO, DKKZ, EMT}, the Camassa-Holm equation \cite{Min}, and the conventional focusing \cite{BKS} and defocusing \cite{BFP} NLS equations.  Particularly, asymmetric boundary conditions arise in nonlinear optics, for describing an input  wave with different amplitude in the two limits, and in hydrodynamics, for modeling  shock waves of temporally nondecreasing intensity. Problems with different backgrounds are known to be a rich source of  nonlinear phenomena, such as modulational instability \cite{OOS, ZO}, asymptotic solitons \cite{H, KK, KMnls}, dispersive shock  waves 
\cite{B18, BKS, BV07, EH16}, rarefaction waves \cite{J}, etc. 

In the present paper we deal with the initial conditions $q_0(x)$ close to the 
``shifted step function'' $q_{R,A}(x)$: 
\begin{equation}\label{shifted-step}
q_{R,A}(x)=\begin{cases}
0, & x<R, \\
A, & x>R,
\end{cases}
\end{equation}
 where $A>0$ and $R>0$ are constants. 
 Notice that in the case of local integrable equations which are translation invariant (e.g., the  NLS equation \cite{BKS}), it is clear that the long-time asymptotics 
of the solution of the initial value problem with these initial conditions
along the rays $\frac{x}{4t}=const$ does not depend on $R$.
 But in the case of a nonlocal equation,  the situation is obviously
different: the nonlocal term(s) immediately ``mixes up'' the state of the system at $x$ and $-x$
and thus one expects the different behavior for different $R$.

The case of the focusing NNLS equation (i.e. $\sigma=1$ in (\ref{1-a}))
is considered in \cite{RSfs}, where we show how sectors with different 
long-time behavior arise depending on $A$ and $R$.
Particularly, in \cite{RSfs} we show that 
 for $R\in\left[0,\frac{\pi}{2A}\right)$, there are two  sectors with different asymptotics: (i) the quarter plane $x<0, t>0$, where the  solution decays to 0 
and (ii) the quarter plane $x>0, t>0$, where the solution approaches the
``modulated constant'' (i.e., different (generally, nonzero) constants along different rays  $x/t=const$). 
 Moreover, if $R\in\left(\frac{(2n-1)\pi}{2A},\frac{(2n+1)\pi}{2A}\right)$ for some $n\in\mathbb{N}$, then  each quarter plane,  $x>0, t>0$ and $x<0, t>0$, 
splits  into $2n+1$ sectors with different asymptotic behavior, the sectors with decay altering the sectors with ``modulated constant'' limits. Thus the 
number of sectors in each quarter plane is always odd.

In the present paper, dealing   the defocusing case,  we show that 
 the asymptotic picture is similar to that
in the focusing case, an important difference being that 
the number of sectors in each quarter plane (with qualitatively different long-time
behavior) is always even. More precisely, 
if $\frac{(n-1)\pi}{A} < R < \frac{n\pi}{A}$ for some $n\in\mathbb{N}$,
and if the initial data is close to $q_{R,A}(x)$ with the parameters
satisfying the inequality above, then the following result holds:
\begin{theorem}\label{cor1}
Under  Assumptions (a)--(c), see  Section \ref{bRH}, on the
 spectral functions associated with the initial data
$q_0(x)$, 
the solution of the initial value problem (\ref{1}), (\ref{2}) has the following 
asymptotics as $t\to +\infty$, qualitatively different in  sectors of the $(x,t)$ plane
specified by ranges of $\xi=\frac{x}{4t}$:
\begin{equation}\label{asq1}
q(x,t)=\left\{
\begin{aligned}
& A\delta^2(0,\xi)
\prod\limits_{s=0}^{m-1}\left(\frac{\xi}{p_{n-s}}\right)^2+o(1),& -\Re p_{n-m}<\xi<\omega_{n-m},\\
&o(1),& -\omega_{n-m}<\xi<\Re p_{n-m}, \\
&\frac{-4\overline{p}_{n-m}^2}{A\overline{\delta^2}(0,-\xi)}
\prod\limits_{s=0}^{m-1}\left(\frac{\overline{p}_{n-s}}
{\xi}\right)^2+o(1),& \Re p_{n-m}<\xi<-\omega_{n-m-1},\\
&o(1),& \omega_{n-m-1}<\xi<-\Re p_{n-m}.
\end{aligned}
\right.
\end{equation}
Here  $m=\overline{0,n-1}$, the function $\delta(0,\xi)$ and 
the numbers $\{p_j\}_1^n$ and $\{\omega_j\}_1^{n-1}$ 
($p_j\in\mathbb C$ with $\Im p_j>0$
and $\omega_j\in\mathbb R$) satisfying
\[
-\infty<\Re p_n < -\omega_{n-1}<\Re p_{n-1} < -\omega_{n-2}<\dots <\Re p_1 <0
\]
are determined in terms of the spectral functions associated with the initial data
$q_0(x)$, see (\ref{delta}) and  Assumptions (a)--(c). 
Particularly,   in the case  $n=1$, the principal asymptotic terms are as in
 Figure \ref{asill}.
\end{theorem}

\begin{figure}
	\begin{minipage}[ht]{0.99\linewidth}
		\centering{\includegraphics[width=0.7\linewidth]{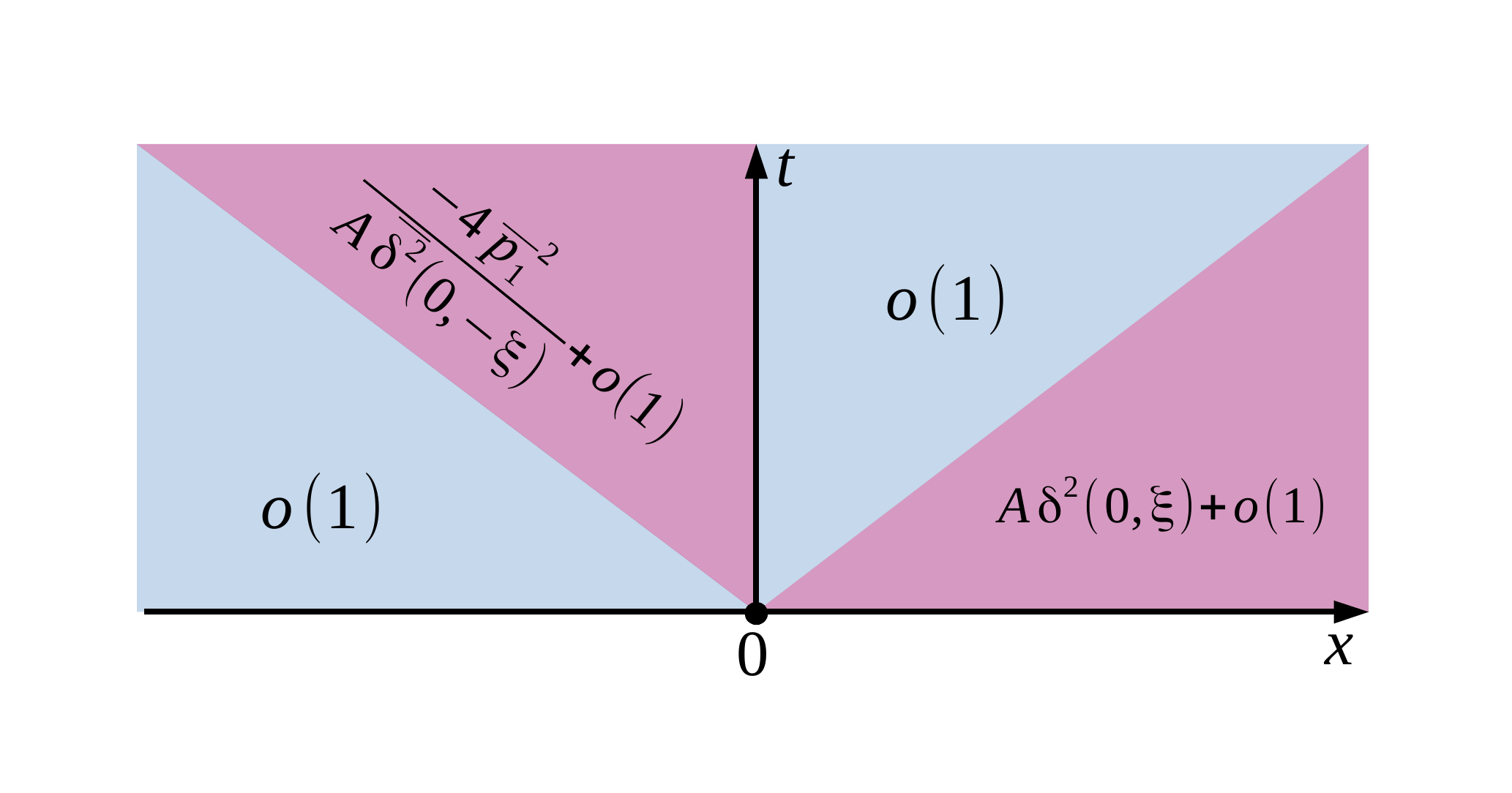}}
		\caption{Asymptotic behavior of the solution of problem (\ref{1}), (\ref{2})
		satisfying Assumptions (a)-(c) with $n=1$.}
		\label{asill}
	\end{minipage}
\end{figure}

Moreover, we are able to make this asymptotics more precise, either 
including in them a second term or writing explicitly a main decaying term
for the corresponding sectors, see Theorem \ref{th1}.

Our main tools used for obtaining these results are
 the inverse scattering transform (IST) method in the form of 
a Riemann-Hilbert (RH) problem and the subsequent use of the 
nonlinear steepest decent method (Deift and Zhou method; see \cite{DZ, DIZ} and   \cite{DVZ94, DVZ97, MM} for its extensions) for the large time analysis of the basic RH problem. Two main peculiarities of the adaptation of this approach 
to our problem are (i) the presence of a singularity on the contour for the 
original RH problem and (ii) the winding of the argument of certain spectral functions leading  to a strong singularity  on the contour 
for the ``deformed RH problem'' (needed for performing the long-time
analysis, see (\ref{windingRH}) and (\ref{p-res-2}) below).

The article is organized as follows. In Section 2 we present the implementation of the inverse scattering transform method for the initial value problem (\ref{1}) in the form of a Riemann--Hilbert problem and examine the properties of the spectral functions associated to the initial data. 
A special attention is payed to the case of ``pure step initial data'',
i.e., for $q(x,0)=q_{R,A}(x)$, since this case provides ingredients guiding the study in the general case.
The asymptotic analysis of the associated Riemann--Hilbert problem and the main result of the paper (Theorem \ref{th1}) are presented in Section 3. In Section 4 we briefly discuss the asymptotics in the transition zones,  and a short conclusion 
is given in Section 5. 

\section{Inverse scattering transform and the Riemann-Hilbert problem}\label{ist}

The implementation of the Riemann-Hilbert problem approach to the step-like problems for local NLS equations substantially differs in the defocusing and focusing cases, due, in particular, to the fact that 
the structure of the spectrum of the associated differential operators 
(from the Lax pair representation) is different: either the whole 
spectrum is located  on the real axis (defocusing case)  or a part of it is 
outside this axis (focusing case). With this respect, we notice that the focusing and defocusing 
variants of the NNLS equation are closer to each other: in the both cases,
(i) there is a point singularity 
on the real axis and (ii) the winding of the argument of certain spectral functions
takes place, which affects the consequent asymptotic analysis.
The next subsection presents the results of the direct scattering analysis 
for the both (focusing and defocusing) cases.

\subsection{Direct scattering}\label{ds}

As we have already mentioned, 
the NNLS equation (\ref{1-a}) is a compatibility condition of the system of two linear differential equations, the so-called  Lax pair  \cite{AMP}:
\begin{equation}
\label{LP}
\left\{
\begin{array}{lcl}
\Phi_{x}+ik\sigma_{3}\Phi=U(x,t)\Phi\\
\Phi_{t}+2ik^{2}\sigma_{3}\Phi=V(x,t,k)\Phi\\
\end{array}
\right.
\end{equation}
where $\sigma_3=\left(\begin{smallmatrix} 1& 0\\ 0 & -1\end{smallmatrix}\right)$, $\Phi(x,t,k)$ is a $2\times2$ matrix-valued function, $k\in\mathbb{C}$ is an auxiliary (spectral) parameter, and the matrix coefficients $U(x,t)$ and $V(x,t,k)$ can be written  in terms of the solution $q(x,t)$ of the NNLS equation as follows:
\begin{equation}
U(x,t)=\begin{pmatrix}
0& q(x,t)\\
-\sigma\bar{q}(-x,t)& 0\\
\end{pmatrix},\qquad 
V=\begin{pmatrix}
V_{11}& V_{12}\\
V_{21}& V_{22}\\
\end{pmatrix},
\end{equation}
where $V_{11}=-V_{22}=i\sigma q(x,t)\bar{q}(-x,t)$, $V_{12}=2kq(x,t)+iq_{x}(x,t)$,
$V_{21}=-2k\sigma\bar{q}(-x,t)+i\sigma(\bar{q}(-x,t))_{x}$.

Taking into account the boundary conditions (\ref{2}) and assuming that the solution $q(x,t)$ of the problem (\ref{1}) exists, we conclude that the matrices $U(x,t)$ and $V(x,t,k)$ converge to the following constant (w.r.t. $x$ and $t$) matrices (cf.  \cite{RSs}):
\begin{equation}
U(x,t)\rightarrow U_{\pm}\ \mbox{and}\ 
V(x,t,k)\rightarrow V_{\pm}(k)\quad \mbox{as}\  x\rightarrow\pm\infty,
\end{equation}
where
\begin{equation}
U_+=
\begin{pmatrix}
0 & A\\
0 & 0
\end{pmatrix},\,
U_-=
\begin{pmatrix}
0 & 0\\
-\sigma A & 0
\end{pmatrix},\,
V_+(k)=
\begin{pmatrix}
0 & 2kA\\
0 & 0
\end{pmatrix},\,
V_-(k)=
\begin{pmatrix}
0 & 0\\
-2\sigma kA & 0
\end{pmatrix}.
\end{equation}
Notice that the system (\ref{LP}) is compatible with $U_+$, $V_+$ or $U_-$, $V_-$ used instead of $U$ and $V$, so the boundary conditions $(\ref{2})$ are meaningful for the NNLS equations (in contrast with their local counterparts, where the boundary conditions 
have to be exact solutions depending on $x$ and $t$).

Introduce the ``background solutions'' $\Phi_{\pm}$
which are the solutions of the system (\ref{LP}) with $U(x,t)$ and $V(x,t,k)$ 
replaced by $U_\pm$ and $V_\pm$ respectively:
\begin{equation}
\label{5}
\Phi_{\pm}(x,t,k)=N_{\pm}(k)e^{-(ikx+2ik^2t)\sigma_3},
\end{equation}
where 
$
N_+(k)=
\begin{pmatrix}
1 & \frac{A}{2ik}\\
0 & 1
\end{pmatrix},\,
N_-(k)=
\begin{pmatrix}
1 & 0\\
\frac{\sigma A}{2ik} & 1
\end{pmatrix}.
$
Similarly to  the case of the focusing NNLS equation \cite{RSs, RSfs}, $N_{\pm}(k)$ have singularities at $k=0$, which play a significant role in the analysis (see the basic Riemann-Hilbert problem in  Section \ref{bRH} below).

Next define the $2\times2$-valued functions $\Psi_j(x,t,k)$, $j=1,2$, $x\in\mathbb{R}$, $t\geq0$ as the solutions of the linear Volterra integral equations:
\begin{subequations}\label{6}
\begin{align}
\label{6-1}
&\Psi_1(x,t,k)=N_-(k)+\int^x_{-\infty}G_-(x,y,t,k)\left(U(y,t)-U_-\right)\Psi_1(y,t,k)e^{ik(x-y)\sigma_3}\,dy,\,
k\in(\mathbb{C}^+,\mathbb{C}^-),\\
\label{6-2}
&\Psi_2(x,t,k)=N_+(k)-\int_x^{\infty}G_+(x,y,t,k)\left(U(y,t)-U_+\right)\Psi_2(y,t,k)e^{ik(x-y)\sigma_3}\,dy,\,
k\in(\mathbb{C}^-,\mathbb{C}^+),
\end{align}
\end{subequations}
where the Cauchy matrices $G_{\pm}(x,y,t,k)$ have the form
\begin{equation}\label{CG}
G_{\pm}(x,y,t,k)=\Phi_{\pm}(x,t,k)[\Phi_{\pm}(y,t,k)]^{-1},
\end{equation}
and $\mathbb{C}^{\pm}=\left\{k\in\mathbb{C}\,|\pm\Im k>0\right\}$.
Here and below, the notation
$k\in(\mathbb{C}^+,\mathbb{C}^-)$ ($k\in(\mathbb{C}^-,\mathbb{C}^+)$), means that the first and the second column of the relevant matrix can be analytically continued from the real axis into respectively the upper (lower) and lower (upper) half-plane as bounded functions. The functions $\Psi_j(x,t,k)$, $j=1,2$, play a significant role in the analysis, namely, they are the key elements in the construction of the basic Riemann-Hilbert problem (see Section \ref{bRH} below). We summarize their main properties (cf. \cite{RSs}) in 
\begin{proposition}\label{psi-prop}
The $2\times2$-valued matrix functions $\Psi_j(x,t,k)$, $j=1,2$ (see \ref{6}) have the following properties:
\begin{enumerate}[(i)]
\item The  functions 
\begin{equation}\label{Jost}
\Phi_j(x,t,k):=\Psi_j(x,t,k)e^{-(ikx+2ik^2t)\sigma_3},\quad k\in\mathbb{R}\setminus\{0\},\quad j=1,2,
\end{equation}
are the Jost solutions of the system (\ref{LP}) satisfying the  boundary conditions (see (\ref{5}))
$$
\Phi_1(x,t,k)\to\Phi_-(x,t,k)\ \ \text{as}\ x\to-\infty,\quad
\Phi_2(x,t,k)\to\Phi_+(x,t,k)\ \ \text{as}\ x\to+\infty.
$$
\item $\det\Psi_j(x,t,k)=1,\quad x,k\in\mathbb{R},\,t\geq0,\quad j=1,2.$
\item The  symmetry property:
\begin{equation}\label{psi-sym}
\Lambda\overline{\Psi_1(-x,t,-k)}\Lambda^{-1}=\Psi_2(x,t,k),\quad k\in\mathbb{R}\setminus\{0\},
\end{equation}
where $\Lambda=\begin{pmatrix}0 & \sigma\\1& 0\end{pmatrix}$.
\item As $k\to0$, the columns $\Psi_j^{(1)}(x,t,k)$ and $\Psi_j^{(2)}(x,t,k)$
of $\Psi_j(x,t,k)$, $j=1,2$ behave as follows:
\begin{subequations}\label{add-symm}
	\begin{align}
	&\Psi_1^{(1)}(x,t,k)=\frac{1}{k}\begin{pmatrix}v_1(x,t)\\ v_2(x,t)\end{pmatrix}+O(1),
	&\Psi_1^{(2)}(x,t,k)=\frac{2i\sigma}{A}\begin{pmatrix}
	v_1(x,t)\\v_2(x,t)\end{pmatrix}+O(k),\\
	&\Psi_2^{(1)}(x,t,k)=-\frac{2i}{A}\begin{pmatrix}
	\sigma\overline{v_2}(-x,t)\\
	\overline{v_1}(-x,t)\end{pmatrix}+O(k),
	&\Psi_2^{(2)}(x,t,k)=-\frac{1}{k}\begin{pmatrix}
	\sigma\overline{v_2}(-x,t)\\\overline{v_1}(-x,t)\end{pmatrix}+O(1),
	\end{align}
\end{subequations}
where $v_1(x,t)$ and $v_2(x,t)$ are some functions.
\end{enumerate}
\end{proposition}
\begin{proof}
Item $(i)$ can be verified directly from the definition of $\Psi_j(x,t,k)$ (see (\ref{6})). Item $(ii)$ follows from the facts that (a) $\det\Phi_j(x,t,k)=\det\Psi_j(x,t,k)$ for $x,k\in\mathbb{R}$ and $t\geq0$ and (b) $U(x,t)$ and $V(x,t)$ are traceless matrices.
Item $(iii)$  follows from the symmetry 
$
\Lambda
\overline{U(-x,t)}
\Lambda^{-1}=U(x,t).
$

Concerning item $(iv)$, we observe that the structure of the singularity of $N_\pm(k)$ as $k\to0$ and the definition of $\Psi_j(x,t,k)$, $j=1,2$ imply that, as $k\to 0$, 
\begin{subequations}
	\label{vw}
	\begin{align}
	\label{vwa}
	&\Psi_1^{(1)}(x,t,k)=\frac{1}{k}\begin{pmatrix}v_1(x,t)\\ v_2(x,t)\end{pmatrix}+O(1),
	&\Psi_1^{(2)}(x,t,k)=\begin{pmatrix}\tilde{v}_1(x,t)\\ \tilde{v}_2(x,t)\end{pmatrix}+O(k),\\
	\label{vwb}
	&\Psi_2^{(1)}(x,t,k)=\begin{pmatrix}\tilde{w}_1(x,t)\\ \tilde{w}_2(x,t)\end{pmatrix}+O(k),
	&\Psi_2^{(2)}(x,t,k)=\frac{1}{k}\begin{pmatrix}w_1(x,t)\\ w_2(x,t)\end{pmatrix}+O(1).
	\end{align}
\end{subequations}
Then, from the symmetry relation (\ref{psi-sym}) it follows that
\begin{equation}
\begin{pmatrix}
w_1(x,t)\\w_2(x,t)
\end{pmatrix}=
\begin{pmatrix}
-\sigma\overline{v_2}(-x,t)\\
-\overline{v_1}(-x,t)
\end{pmatrix}\quad\text{and}\quad
\begin{pmatrix}\tilde{w}_1(x,t)\\ \tilde{w}_2(x,t)
\end{pmatrix}=
\begin{pmatrix}
\overline{\tilde{v}_2}(-x,t)\\
\sigma\overline{\tilde{v}_1}(-x,t)
\end{pmatrix}.
\end{equation}
Further, substituting (\ref{vwa})  into 
(\ref{6-1})  we conclude that
$v_j(x,t)$, $j=1,2$ satisfy the system of integral equations
\begin{equation}
\label{vv}
\begin{cases}
v_1(x,t)=\int_{-\infty}^{x}q(y,t)v_2(y,t)\,dy,\\
v_2(x,t)=-i\sigma\frac{A}{2}-\sigma\int_{-\infty}^{x}
\overline{q(-y,t)}v_1(y,t)\,dy.
\end{cases}
\end{equation}
whereas 
$\tilde{v}_j(x,t)$, $j=1,2$ solve the following system of equations:
\begin{equation}\label{2.14}
\begin{cases}
\tilde{v}_1(x,t)=\int_{-\infty}^{x}q(y,t)\tilde{v}_2(y,t)\,dy,\\
\tilde{v}_2(x,t)=1-\sigma\int_{-\infty}^{x}
\overline{q(-y,t)}\tilde{v}_1(y,t)\,dy.
\end{cases}
\end{equation}
Comparing (\ref{vv}) with (\ref{2.14}) it follows that 
$$
\begin{pmatrix}
\tilde{v}_1(x,t)\\
\tilde{v}_2(x,t)
\end{pmatrix}=\frac{2i\sigma}{A}
\begin{pmatrix}
v_1(x,t)\\v_2(x,t)
\end{pmatrix}.
$$
\end{proof}

Since the Jost solutions $\Phi_1(x,t,k)$ and $\Phi_2(x,t,k)$ defined by (\ref{Jost}) satisfy the system (\ref{LP}) for all 
$k\in\mathbb{R}\setminus\{0\}$, they are related by a matrix function independent of $x$ and $t$:
\begin{equation}
\Phi_1(x,t,k)=\Phi_2(x,t,k)S(k),\quad k\in\mathbb{R}\setminus\{0\},
\end{equation}
or, in terms of $\Psi_j(x,t,k)$, $j=1,2$
\begin{equation}
\label{9}
\Psi_1(x,t,k)=\Psi_2(x,t,k)e^{-(ikx+2ik^2t)\sigma_3}S(k)e^{(ikx+2ik^2t)\sigma_3},\,\,k\in\mathbb{R}\setminus\{0\},
\end{equation}
where $S(k)$ is the so-called scattering matrix. Due to the symmetry (\ref{psi-sym}) the matrix $S(k)$ can be written as follows (cf. \cite{RSs})
\begin{equation}
S(k)=
\begin{pmatrix}
a_1(k)& -\sigma\overline{b(-k)}\\
b(k)& a_2(k)
\end{pmatrix},\qquad k\in\mathbb{R}\setminus\{0\},
\end{equation}
with some functions (the so-called spectral functions) $b(k)$ and $a_j(k)$, $j=1,2$, which satisfy the symmetry relation $\overline{a_j(-k)}=a_j(k)$, $j=1,2$. The spectral functions can be obtained in terms of the initial data only:
\begin{equation}
S(k)=\Psi_2^{-1}(0,0,k)\Psi_1(0,0,k),
\end{equation}
or, alternatively, by using the determinant relations
\begin{subequations}\label{sd}
	\begin{align}
	b(k)&=\det\left(\Psi_2^{(1)}(0,0,k),\Psi_1^{(1)}(0,0,k)\right),
	\quad k\in\mathbb{R},\\
	a_1(k)&=\det\left(\Psi_1^{(1)}(0,0,k),\Psi_2^{(2)}(0,0,k)\right),
	\quad k\in\overline{\mathbb{C}^{+}}\setminus\{0\},\\
	a_2(k)&=\det\left(\Psi_2^{(1)}(0,0,k),\Psi_1^{(2)}(0,0,k)\right),
	\quad k\in\overline{\mathbb{C}^{-}}.
	\end{align}
\end{subequations}

We summarize the properties of the spectral functions $b(k)$ and $a_j(k)$, $j=1,2$ in the following proposition (cf. \cite{RSs}; particularly, item 5 below follows from
(\ref{add-symm})):
\begin{proposition}\label{properties}
The spectral functions $b(k)$, $a_j(k)$, $j=1,2$, have the following properties
\begin{enumerate}
\item 
$a_{1}(k)$ is analytic in  $k\in\mathbb{C}^{+}$
and continuous in 
$\overline{\mathbb{C}^{+}}\setminus\{0\}$;
$a_{2}(k)$ is analytic in $k\in\mathbb{C}^{-}$
and continuous in 
$\overline{\mathbb{C}^{-}}$. 
\item
$a_{j}(k)=1+{O}\left(\frac{1}{k}\right)$, $j=1,2$ and 
$b(k)={O}\left(\frac{1}{k}\right)$ as $k\rightarrow\infty$ (the latter holds
for $k\in{\mathbb R}$).
\item
$\overline{a_{1}(-\bar{k})}=a_1(k)$,  
$k\in\overline{\mathbb{C}^{+}}\setminus\{0\}$; \qquad
$\overline{a_{2}(-\bar{k})}=a_2(k)$,  
$k\in\overline{\mathbb{C}^{-}}$.
\item
$a_{1}(k)a_{2}(k)+\sigma b(k)\overline{b(-
\bar{k})}=1$, $k\in{\mathbb R}\setminus\{0\}$ (follows from $\det S(k)=1$).
\item As $k\to 0$, $a_1(k)=\sigma\frac{A^2a_2(0)}{4k^2}+O\left(\frac{1}{k}\right)$ and 
$b(k)=\sigma\frac{Aa_2(0)}{2ik}+O\left(1\right)$.
\end{enumerate}
\end{proposition}

\subsection{Spectral functions for the ``shifted step'' initial data}

Henceforth, we deal with the defocusing NNLS equation (making comparisons, if appropriate,
with the case of 
 the focusing NNLS equation). Analytic considerations in the case
of pure step initial data (\ref{shifted-step}) presented below 
will guide us in establishing the 
general framework suitable for the asymptotic analysis.

The spectral functions associated with the  initial data $q_0(x)=q_{R,A}(x)$ 
are given explicitly by 
\begin{subequations}
\label{spf}
\begin{align}
&a_1(k)=1- \frac{A^2}{4k^2}e^{4ikR},\\
\label{a_2_zeros}
&a_2(k)=1,\\
&b(k)=-\frac{A}{2ik}e^{2ikR}.
\end{align}
\end{subequations}
Indeed, the scattering matrix $S(k)$ can be obtained from (\ref{9}) taking $x=-R$ and $t=0$:
\begin{equation}\label{scatt-st}
S(k)=e^{-ikR\sigma_3}\Psi_2^{-1}(-R,0,k)\Psi_1(-R,0,k)e^{ikR\sigma_3}.
\end{equation}
It follows from  (\ref{6}) at $t=0$ that
\begin{subequations}
\label{psiR}
\begin{align}
&\Psi_1(-R,0,k)=N_-(k),\\
&\Psi_2(-R,0,k)=N_+(k)-\int^R_{-R}G_+(-R,y,0,k)
\begin{pmatrix}
0& -A\\
0& 0
\end{pmatrix}
\Psi_2(y,0,k) e^{-ik(R+y)\sigma_3}\,dy,
\end{align}
\end{subequations}
where 
$\Psi_2(x,0,k)$ for $x\in\left[-R,R\right]$ solves the  integral equation
\begin{equation}\label{int-psi2}
\Psi_2(x,0,k)=N_+(k)-\int^R_x G_+(x,y,0,k)
\begin{pmatrix}
0& -A\\
0& 0
\end{pmatrix}
\Psi_2(y,0,k)
e^{ik(x-y)\sigma_3}\,dy,\quad x\in[-R, R].
\end{equation}
From the definition of $G_+$ (see (\ref{CG})) it follows that
\begin{equation*}
G_+(x,y,0,k)=
\begin{pmatrix}
e^{-ik(x-y)}& \frac{A}{2ik}\left(e^{ik(x-y)}-e^{-ik(x-y)}\right)\\
0& e^{ik(x-y)}
\end{pmatrix},
\end{equation*}
and then direct calculations show that the solution of the equation (\ref{int-psi2}) is as follows:
\begin{equation}\label{psi2R}
\Psi_2(x,0,k)=
\begin{pmatrix}
1 & \frac{A}{2ik}e^{2ik(R-x)}\\
0 & 1
\end{pmatrix},\quad x\in\left[-R,R\right].
\end{equation}
Substituting (\ref{psiR}) and (\ref{psi2R}) into (\ref{scatt-st}) we arrive at (\ref{spf}).

Now let us analyze the locations of zeros of $a_1(k)$ in $\overline{{\mathbb{C}^{+}}}$ and the  behavior of its argument for $k\in\mathbb{R}$.

\begin{proposition}\label{a_1ss}

\begin{enumerate}[(i)]
\item For $\frac{(n-1)\pi}{A}<R<\frac{n\pi}{A}$, $n\in\mathbb{N}$, $a_1(k)$ has the following properties:
\begin{itemize}
\item 
$a_1(k)$ has $2n$ simple zeros in $\overline{\mathbb{C}^{+}}$:
$\{p_j, -\overline{p}_j\}_{j=1}^{n}$. 
Here $\{\Re p_j\}_{j=1}^{n}$ are the ordered set of solutions of the equations 
\begin{equation}
\label{transc2}
k=\pm\frac{A}{2}\cos (2kR) e^{-2kR\tan (2kR)},
\end{equation}
considered for $k<0$, and $\Im p_j$ and $\Re p_j$ are related by
\begin{equation}\label{Imp_j}
\Im p_j=\Re p_j\tan (2\Re p_jR),\quad j=\overline{1,n}.
\end{equation}
Notice that 
\begin{equation}
\label{interv}
\Re p_j\in\left(-\frac{(2j-1)\pi}{4R},-\frac{(j-1)\pi}{2R}\right),
\quad j=\overline{1,n}
\end{equation}
and that behavior of the real and imaginary parts of the zeros $p_j$ as $R$ increases is similar to the case of the focusing NNLS equation \cite{RSfs}.
\item
Define $\omega_j$, $j=\overline{0,n}$ as follows: $\omega_0=0$, $\omega_j=\frac{j\pi}{2R}$ for $j=\overline{1,n-1}$, and $\omega_{n}=\infty$. Then  
\begin{subequations}\label{winding}
\begin{align}
&\int_{-\infty}^{-\omega_{n-j}}d\arg a_1(k) = (2j-1)\pi,\quad j=\overline{1,n-1},\\
&\int_{-\infty}^{-\xi}d\arg a_1(k)\in((2j-1)\pi,(2j+1)\pi),\quad -\omega_{n-j}<-\xi<-\omega_{n-j-1},\quad j=\overline{0,n-1}.
\end{align}
\end{subequations}

\end{itemize}

\item If $R=\frac{n\pi}{A}$ for some $n\in\mathbb{N}\cup\{0\}$, then  
$a_1(k)$ has $2n+2$ simple zeros in $\overline{\mathbb{C}^{+}}$ at 
$\{\pm\frac{A}{2}$,  $\{p_j, -\overline{p}_j\}_{j=1}^{n}\}$, where   $\Re p_j$ ($j=\overline{1,n}$) are the solutions of (\ref{transc2}), and $\Im p_j$ are determined by (\ref{Imp_j}).
\end{enumerate}
\end{proposition}

\begin{proof}

 Observe that the equation $a_1(k)=0$ is equivalent to the system
\begin{equation}
\label{syst}
\begin{cases}
k_1=\pm\frac{A}{2}\cos(2k_1R)e^{-2k_2R}\\
k_2=\pm\frac{A}{2}\sin(2k_1R)e^{-2k_2R}
\end{cases},
\end{equation}
where $k=k_1+ik_2$, $k\in\overline{\mathbb{C}^{+}}\setminus\{0\}$. Due to the symmetry relation $a_1(k)=\overline{a_1(-\bar{k})}$ it is sufficient to consider (\ref{syst}) 
for $k_1\geq0$ only.

(i) For $k_1=0$, the system (\ref{syst}) clearly has no solutions, so $a_1(k)$ has no purely imaginary zeros (notice that in the focusing case, $a_1(k)$ has one simple purely imaginary zero \cite{RSfs}).

(ii)  Assuming $k_2=0$, the second equation in (\ref{syst}) implies that $k_1$ must be equal to $\frac{\pi n}{2R}$ with some $n\in\mathbb{N}$.
But then,  from the first equation in (\ref{syst}) we conclude that $k_1=\frac{A}{2}$. Therefore, $k=\pm\frac{A}{2}$ are simple zeros of $a_1(k)$ if and only if there exists $n\in\mathbb{N}$ such that 
$\pi n=AR$. Notice that in the case $R=0$, the spectral function $a_1(k)$ has exactly two simple zeros,  $\frac{A}{2}$ and $-\frac{A}{2}$.

(iii) Now, let's look at the location of  zeros of $a_1(k)$ in the open quarter plane $k_1>0$, $k_2>0$. Dividing the  equations in (\ref{syst}) sidewise we arrive at (cf. (\ref{Imp_j}))
\begin{equation}\label{cot}
k_2=k_1\tan(2k_1R),\quad k_1\neq\frac{\pi(2n+1)}{4R},\,
n\in\mathbb{N},
\end{equation}
from which  we conclude (cf. (\ref{interv})) that 
\begin{equation}\label{domain}
k_1\in\left(\frac{(n-1)\pi}{2R},\frac{(2n-1)\pi}{4R}\right),\quad n\in\mathbb{N}.
\end{equation}
Substituting (\ref{cot}) into the first equation in (\ref{syst}) and taking into account the sign of $\cos(2k_1R)$ for $k_1$ satisfying (\ref{domain}), we obtain the equations 
for $k_1$:
\begin{subequations}\label{k_1}
\begin{equation}\label{k1-1}
k_1=\frac{A}{2}\cos(2k_1R)e^{-2k_1R\tan(2k_1R)}\quad \text{for}\ 
k_1\in\left(\frac{(n-1)\pi}{R},\frac{(4n-3)\pi}{4R}\right),\quad n\in\mathbb{N},
\end{equation}
or 
\begin{equation}\label{k1-2}
k_1=-\frac{A}{2}\cos(2k_1R)e^{-2k_1R\tan(2k_1R)} \quad \text{for}\ 
k_1\in\left(\frac{(2n-1)\pi}{2R},\frac{(4n-1)\pi}{4R}\right),\quad n\in\mathbb{N}.
\end{equation}
\end{subequations}
Since the r.h.s. of (\ref{k1-1}) and (\ref{k1-2}) monotonically decrease 
w.r.t. $k_1$
in the  corresponding intervals, it follows that for 
$\frac{(n-1)\pi}{A}<R\leq\frac{n\pi}{A}$ equations (\ref{k_1}) 
have $n$ simple solutions $\{k_{1,j}\}_{j=1}^{n}$ in the quarter plane $k_1>0$, $k_2>0$ such that  $k_{1,j}\in\left(\frac{(j-1)\pi}{2R},\frac{(2j-1)\pi}{4R}\right)$, 
$j=\overline{1,n}$ (cf.  (\ref{interv})).

Concerning the   winding properties of $\arg a_1(k)$, 
the estimates (for $\frac{(n-1)\pi}{A}<R<\frac{n\pi}{A}$)
\begin{subequations}
\begin{align}
&\frac{A^2}{4k_{(m)}^2}e^{4ik_{(m)}R}<1\quad \text{for}\ 
k_{(m)}=-\frac{m\pi}{2R},\quad
m\in\mathbb{N},\quad
m\geq n,\\
&\frac{A^2}{4k_{(m)}^2}e^{4ik_{(m)}R}>1 \quad \text{for}\ 
k_{(m)}=-\frac{m\pi}{2R},\quad
m\in\mathbb{N},\quad
m< n,
\end{align}
\end{subequations}
yield (\ref{winding}).
\end{proof}
\begin{remark}
\label{cdnls}
Considering the pure step initial data $q_{R,A}$ as varying with $R$ (for a fixed $A$),
the values $R=\frac{n\pi}{A}$, $n=0,1,2,\dots$ turn to be the bifurcation points: when $R$ is passing any of these values,  $a_1(k)$ acquires an additional pair of zeros 
at $k=\pm\frac{A}{2}$
(cf. \cite{BP}, Section 4.1, where the box-type piecewise-constant initial data
for the defocusing NLS equation with nonzero boundary conditions at infinity 
are considered illustrating the bifurcation of discrete eigenvalues).
\end{remark}
\subsection{The basic Riemann-Hilbert problem and inverse scattering}
\label{bRH}

One of the main advantages of  the Riemann-Hilbert approach in the Inverse Scattering Transform method 
is that it is highly efficient  in the asymptotic analysis.
Recall that  the  Riemann-Hilbert (RH) problem 
(as widely used in applications to integrable systems)
consists in finding an $n\times n$-valued piece-wise meromorphic function that satisfies a prescribed jump condition across a contour in the complex plane and prescribed conditions at singular points (if any). The jump matrix for  RH problems associated with   initial (and initial-boundary) value problems for integrable systems are usually oscillatory with respect to a large parameter (in our case, time); 
in  
treating (asymptotically) these problems,  
 the so-called nonlinear steepest decent  method (Deift and Zhou method \cite{DZ}) has
proved to be extremely efficient. 

The construction of the RH problem for an integrable system is usually based
(at least in the case when the  differential equations in the Lax pair representation
are of the second order), 
 on analytic properties of the Jost solutions $\Phi_j(x,t,k)$ and, correspondingly,
the functions $\Psi_j(x,t,k)$. Similarly
to the case of the focusing NNLS equation \cite{RSs, RSfs},
we
 define the $2\times 2$-valued, piece-wise meromorphic (relative to the real line) function $M(x,t,k)$ by
\begin{equation}
\label{DM}
M(x,t,k)=
\left\{
\begin{array}{lcl}
\left(\frac{\Psi_1^{(1)}(x,t,k)}{a_{1}(k)},\Psi_2^{(2)}(x,t,k)\right),\quad k\in\mathbb{C}^+,\\
\left(\Psi_2^{(1)}(x,t,k),\frac{\Psi_1^{(2)}(x,t,k)}{a_{2}(k)}\right),\quad k\in\mathbb{C}^-.\\
\end{array}
\right.
\end{equation}
The scattering relation (\ref{9}) implies that the boundary values 
$M_\pm(x,t,k) = \underset{k'\to k, k'\in {\mathbb C}^\pm}{\lim} M(x,t,k')$, $k\in\mathbb R$ (we take the non-tangential limits)
 satisfy the  jump condition
\begin{equation}\label{jr}
M_+(x,t,k)=M_-(x,t,k)J(x,t,k),\qquad k\in\mathbb{R}\setminus\{0\}
\end{equation}
with the  jump matrix
\begin{equation}\label{jump}
J(x,t,k)=
\begin{pmatrix}
1-r_{1}(k)r_{2}(k)& -r_{2}(k)e^{-2ikx-4ik^2t}\\
r_1(k)e^{2ikx+4ik^2t}& 1
\end{pmatrix},
\end{equation}
where the reflection coefficients $r_1(k)$, $j=1,2$ are defined by
\begin{equation}\label{r12}
r_1(k):=\frac{b(k)}{a_1(k)},\quad r_2(k):=\frac{\overline{b(-k)}}{a_2(k)}.
\end{equation}

Observe that from the determinant relation (see item 4 with $\sigma=-1$ in Proposition \ref{properties}) we have
\begin{equation}\label{scalrel}
1-r_1(k)r_2(k)=\frac{1}{a_1(k)a_2(k)}.
\end{equation}
Moreover, 
\begin{equation}
M(x,t,k)\to I,\qquad k\to\infty,
\end{equation}
where  $I$ is the $2\times2$ identity matrix.

Taking into account the singularities of 
$\Psi_j(x,t,k)$, $j=1,2$ and  $a_1(k)$  at  $k=0$  (see Propositions \ref{psi-prop} and \ref{properties}), the function $M(x,t,k)$ has the following behavior as $k\to0$:
\begin{subequations}\label{z}
\begin{align}
\label{+i0}
& M(x,t,k)=
\begin{pmatrix}
-\frac{4}{A^2a_2(0)}v_1(x,t)& \overline{v_2}(-x,t)\\
-\frac{4}{A^2a_2(0)}v_2(x,t)& -\overline{v_1}(-x,t)
\end{pmatrix}
(I+O(k))
\begin{pmatrix}
k& 0\\
0& \frac{1}{k}
\end{pmatrix},& k\rightarrow 0,\quad k\in \mathbb{C}^+,\\
\label{-i0}
& M(x,t,k)=\frac{2i}{A}
\begin{pmatrix}
\overline{v_2}(-x,t)& -\frac{v_1(x,t)}{a_2(0)}\\
-\overline{v_1}(-x,t)& -\frac{v_2(x,t)}{a_2(0)}
\end{pmatrix}
+O(k), & k\rightarrow 0, \quad k\in \mathbb{C}^-,
\end{align}
\end{subequations}
where $v_j(x,t)$, $j=1,2$ are some functions.

Similarly to the focusing NNLS equation \cite{RSfs}, the properties of $a_j(k)$, $j=1,2$ in the case of the ``shifted step'' initial data (see Proposition \ref{a_1ss}) guide us to make  assumptions on the spectral functions $a_j(k)$, $j=1,2$ 
associated with 
 initial data  satisfying (\ref{shifted-step-gen}). We emphasize that these assumptions differ from those made in the focusing case 
(particularly, the order of $\Re p_j$ and $\omega_j$ is different, see (\ref{order}) below), which significantly  affects  the resulting asymptotic formulas.
  
\begin{description}
\item [Assumptions:] 
\item{\textbf{(a)}} 
$a_1(k)$ has $2n$, $n\in\mathbb{N}$, simple zeros 
in $\overline{{\mathbb C}^+}\setminus\{0\}$,
 $\{p_j\}_{j=1}^{n}$ and  $\{-\overline{p}_j\}_{j=1}^{n}$, with $\Im p_j>0$ and $\Re p_n<\dots<\Re p_1<0$. 
\item{\textbf{{(b})}}
$a_2(k)$ has no zeros in $\overline{\mathbb{C}^-}$.
\item{\textbf{(c)}} 
There exist numbers $\omega_m>0$, $m=\overline{1,n-1}$ such that
\begin{equation}\label{order}
-\infty<\Re p_{n}<-\omega_{n-1}<\Re p_{n-1}<-\omega_{n-2}<\dots<\Re p_1<0,
\end{equation}
\begin{subequations}\label{windingRH}
\begin{equation}
\int_{-\infty}^{-\omega_{n-m}}d\arg\left(a_1(k)a_2(k)\right) = (2m-1)\pi,
	\quad m=\overline{1,n-1},
\end{equation}
and 
\begin{equation}\label{xiint}
\int_{-\infty}^{-\xi}d\arg\left(a_1(k)a_2(k)\right)
\in((2m-1)\pi,(2m+1)\pi),\quad -\omega_{n-m}<-\xi<-\omega_{n-m-1},\quad m=\overline{0,n-1}
\end{equation}
(here we adopt the notations $\omega_0=0$ and  $\omega_n=+\infty$).
\end{subequations}
\end{description}

The construction of  $M$ implies that 
at the zeros of $a_1(k)$, $M(x,t,k)$
satisfies the following residue conditions:
\begin{subequations}\label{resin}
\begin{align}
\underset{k=p_j}{\operatorname{Res}} M^{(1)}(x,t,k)&=
\frac{\eta_j}{\dot a_1(p_j)}
e^{2ip_jx+4ip_j^2t}M^{(2)}(x,t,p_j),\quad j=\overline{1,n},\\
\underset{k=-\overline{p}_j}{\operatorname{Res}} M^{(1)}(x,t,k)&=
\frac{1}{\bar\eta_j\dot a_1(-\overline{p}_j)}
e^{-2i\overline{p}_jx+4i\overline{p}_j^2t}M^{(2)}(x,t,-\overline{p}_j),
\quad j=\overline{1,n}.
\end{align}
\end{subequations}
Here $\eta_j$, $j=\overline{1,n}$ are constants 
determined by the initial data through 
 $\Psi_1^{(1)}(0,0,p_j)=\eta_j\Psi_2^{(2)}(0,0,p_j)$.

Basing on the analytic properties of $M$ presented above, we observe that 
we can \emph{characterize} $M$ as the solution of a Riemann--Hilbert
problem, with data uniquely determined by the initial data $q_0(x)$, 
in terms of the associated spectral data.

\begin{description}
\item [Basic Riemann--Hilbert Problem:] 
Given $b(k)$, $k\in{\mathbb R}$ and $a_j(k)$, $j=1,2$ which satisfy  properties 1-5 
in Proposition \ref{properties} and  assumptions (a)-(c) above, 
and constants $\eta_j$, $j=\overline{1,n}$, 
find the $2\times 2$-valued, piece-wise (relative to 
$\mathbb{R}$)
meromorphic in $k$ function $M(x,t,k)$ satisfying the following conditions:
\item [(1)] The jump condition:
\begin{equation}\label{jRH}
M_+(x,t,k)=M_-(x,t,k)J(x,t,k),\qquad k\in\mathbb{R}\setminus\{0\}
\end{equation}
with the  jump matrix $J(x,t,k)$ given by (\ref{jump}), where $r_j(k)$, $j=1,2$
are determined in terns of  $b(k)$ and $a_j(k)$, $j=1,2$
by (\ref{r12}).
\item [(2)] The residue conditions (\ref{resin}).
\item [(3)] The \textit{pseudo-residue} conditions (\ref{z}) at $k=0$, where $v_j(x,t)$, $j=1,2$ are not prescribed.
\item[(4)] The normalization condition at $k=\infty$:
$$
M(x,t,k)=I+O(k^{-1})\quad \mbox{uniformly as } k\to\infty.
$$
\end{description}

Assuming that the Riemann-Hilbert problem (1)--(4) has a solution $M(x,t,k)$, the solution of the initial value problem (\ref{1}), (\ref{2}) can be expressed as follows:
\begin{equation}\label{sol}
q(x,t)=2i\lim_{k\to\infty}kM_{12}(x,t,k)
\end{equation}
or
\begin{equation}\label{sol1}
q(-x,t)=-2i\lim_{k\to\infty}k\overline{M_{21}(x,t,k)}.
\end{equation}


\begin{remark}
The solution of Basic Riemann-Hilbert Problem is unique. Indeed, let $M$ and $\tilde{M}$ be two solutions of the problem, then $M\tilde{M}^{-1}$ has no jump across $\mathbb{R}\setminus\{0\}$ and it is bounded at $k=0$ (which can be seen from (\ref{z})), 
$k=p_j$ and $k=-\bar{p}_j$. Taking into account the normalization condition \textbf{(4)}, by the Liouville theorem it follows that $M\tilde{M}^{-1}\equiv I$.
\end{remark}
\begin{remark}\label{rem-sym}
(\ref{sol}) and (\ref{sol1}) imply that 
the solution of problem 
  (\ref{1}), (\ref{2}) for all $x\in (-\infty,\infty)$
	can be expressed in terms of the solutions of the RH problems evaluated for
	$x\ge 0 $ only.
\end{remark}
\begin{remark}
In contrast with local integrable equations, where zeros of certain spectral functions
(analogues of $a_j(k)$, $j=1,2$) are associated with solitons traveling on a prescribed background
(even in the cases when the background is nonzero), for nonlocal equations, certain number of zeros of 
$a_j(k)$, $j=1,2$ is always associated with the background itself.
In the present paper, we restrict ourselves to the case without additional zeros 
(associated with the deflection of $q_0(x)$ from the background (\ref{shifted-step}).
\end{remark}

\section{The long-time asymptotics}\label{asympt}

In the present section we study the long-time behavior of the solution of the 
initial value  problem 
(\ref{1}), (\ref{2}) under Assumptions \textbf{(a)-(c)}. For this purpose,
 we adapt the nonlinear steepest-decent method \cite{DZ} to  Basic Riemann-Hilbert Problem \textbf{(1)--(4)} (see Section \ref{bRH}).

\subsection{Jump factorizations}

Introducing the phase function
\begin{equation}\label{theta}
\theta(k,\xi)=4k\xi+2k^2
\end{equation}
in terms of the slow variable  $\xi=\frac{x}{4t}$,  the jump matrix (\ref{jump}) 
admits the triangular factorizations  of
two types \cite{RSs, RSfs}:
\begin{subequations}\label{tr}
\begin{align}
\label{tr1}
J(x,t,k)&=
\begin{pmatrix}
1& 0\\
\frac{r_1(k)}{1- r_1(k)r_2(k)}e^{2it\theta}& 1\\
\end{pmatrix}
\begin{pmatrix}
1- r_1(k)r_2(k)& 0\\
0& \frac{1}{1- r_1(k)r_2(k)}\\
\end{pmatrix}
\begin{pmatrix}
1& \frac{-r_2(k)}{1- r_1(k)r_2(k)}e^{-2it\theta}\\
0& 1\\
\end{pmatrix}
\\
\label{tr2}
&=
\begin{pmatrix}
1& -r_2(k)e^{-2it\theta}\\
0& 1\\
\end{pmatrix}
\begin{pmatrix}
1& 0\\
r_1(k)e^{2it\theta}& 1\\
\end{pmatrix}.
\end{align}
\end{subequations}
The idea of the nonlinear steepest descent method is to transform the original RH problem
to such a form, where  the jump matrix converges rapidly to $I$ away from a vicinity of the stationary phase point $k=-\xi$ of $\theta$. 
Since $\theta(k,\xi)$ and its signature table (see Figure \ref{signtable}) are the same as in the case of the local NLS equation,
we can initiate the RH problem transformations similarly to the local case, 
introducing an auxiliary scalar function $\delta(k,\xi)$ in order to 
 get rid of the diagonal factor in (\ref{tr1}). 
This function can be   defined as  the solution of the following scalar RH problem:
\begin{subequations}\label{RHd}
\begin{align}
&\delta_+(k,\xi)=\delta_-(k,\xi)(1-r_1(k)r_2(k)),&&k\in(-\infty,-\xi),\\
&\delta(k,\xi)\rightarrow 1,&&k\rightarrow\infty.
\end{align}
\end{subequations}

Although  problem (\ref{RHd}) seems to be exactly the same as in the case of the local NLS \cite{DIZ}, a principal difference is that the jump $(1-r_1(k)r_2(k))$ 
is, in general,  a complex-valued function, which can lead to a strong singularity 
at $k=-\xi$. 
In order to cope with the similar problem in the focusing case, 
in \cite{RSfs} we introduced a finite number of ``partial functions delta'', which have weak singularities, and proceed with  their  product.  
Here we proceed in a different way,   defining a single  function $\delta(k,\xi)$ as the solution of the scalar RH problem (\ref{RHd}), and then dealing with the strong singularity at $k=-\xi$.

\begin{figure}
\begin{minipage}[h]{0.99\linewidth}
	\centering{\includegraphics[width=0.4\linewidth]{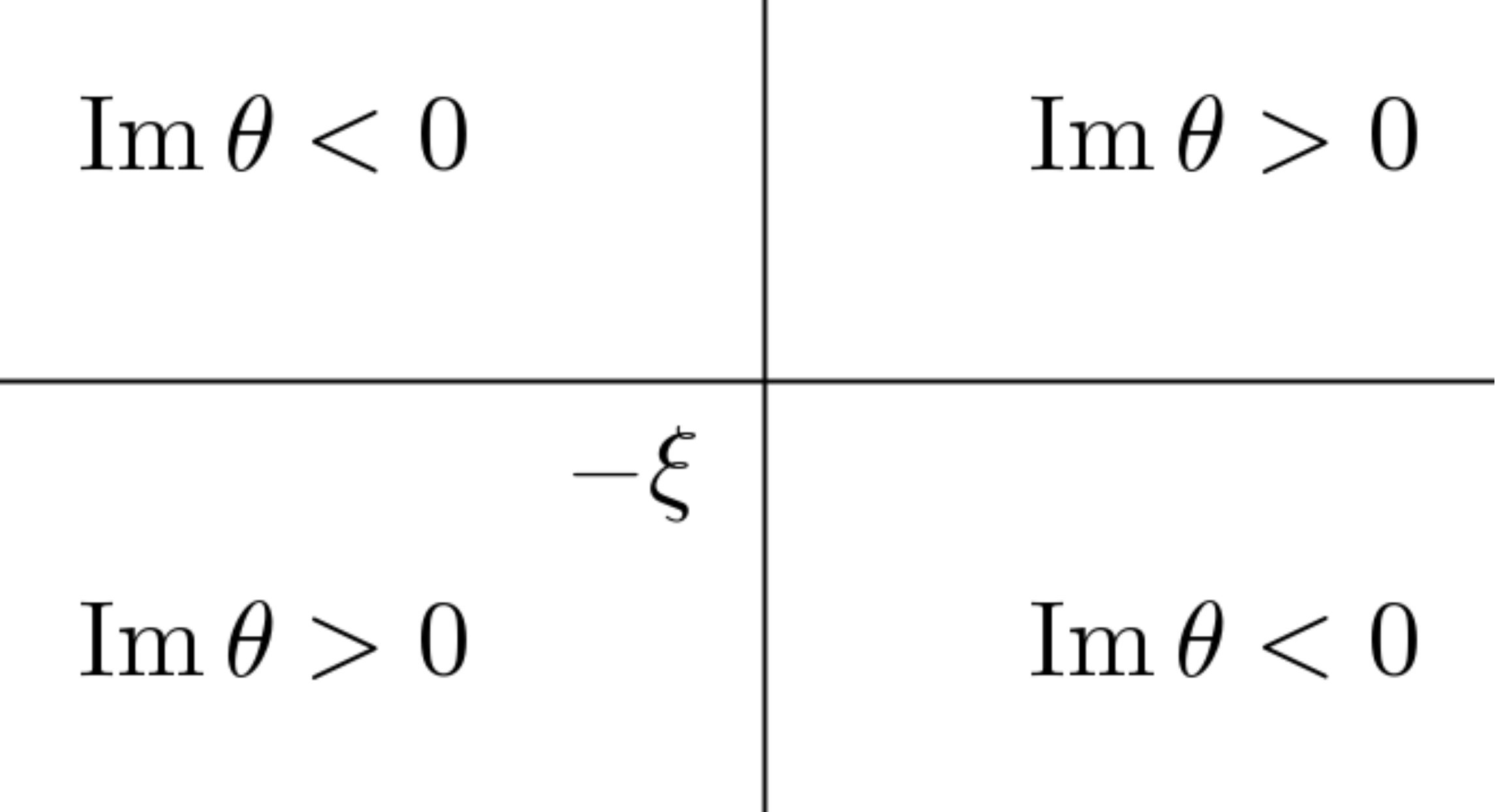}}
	\caption{Signature table}

	\label{signtable}
\end{minipage}
\end{figure}

The function $\delta(k,\xi)$ given by the Cauchy integral
\begin{equation}\label{delta}
\delta(k.\xi)=\exp\left\{
\frac{1}{2\pi i}\int_{-\infty}^{-\xi}
\frac{\ln(1-r_1(\zeta)r_2(\zeta))}{\zeta-k}\,d\zeta
\right\}
\end{equation}
satisfies (\ref{RHd}) and can be written as
\begin{equation}\label{deltasing}
\delta(k,\xi)=(k+\xi)^{i\nu(-\xi)}e^{\chi(k,\xi)},
\end{equation}
where 
\begin{equation}
\chi(k,\xi)=-\frac{1}{2\pi i}\int_{-\infty}^{-\xi}\ln(k-\zeta)
d_{\zeta}(1-r_1(\zeta)r_2(\zeta)),
\end{equation}
and
\begin{equation}
\nu(-\xi)=-\frac{1}{2\pi}\ln|1-r_1(-\xi)r_2(-\xi)|-\frac{i}{2\pi}
\left(
\int_{-\infty}^{-\xi}d\arg(1-r_1(\zeta)r_2(\zeta))
\right).
\end{equation}
Now we notice that in view of (\ref{scalrel}) and (\ref{xiint}) we have:
\begin{subequations}\label{3.8}
\begin{align}
&\Im\nu(-\xi)\in((m-1/2), (m+1/2))\ \ \text{for}\ 
-\xi\in(-\omega_{n-m},-\omega_{n-m-1}),\quad &m=\overline{0,n-1},\\
&\Im\nu(-\omega_{n-m-1})=m+1/2,\quad &m=\overline{0,n-2},
\end{align}
\end{subequations}
which leads to, generally, a strong singularity of $\delta(k,\xi)$, see (\ref{deltasing}).

Introducing
\begin{equation}
\tilde{M}(x,t,k)=M(x,t,k)
\delta^{-\sigma_3}(k,\xi),
\end{equation} 
the function $\tilde{M}(x,t,k)$ satisfies the  jump and norming conditions
\begin{subequations}\label{tildeM}
\begin{align}
\label{14.1}
&\tilde{M}_+(x,t,k)=\tilde{M}_-(x,t,k)\tilde{J}(x,t,k), && k\in\mathbb{R}\setminus\{0\},\\
\label{14.1-norm}
&\tilde{M}(x,t,k)\rightarrow I, &&k\rightarrow\infty
\end{align}
\end{subequations}
with
\begin{equation}
\label{as4}
\tilde{J}(x,t,k)=
\begin{cases}
\begin{pmatrix}
1& 0\\
\frac{r_1(k)\delta_-^{-2}(k,\xi)}{1-r_1(k)r_2(k)}e^{2it\theta}& 1\\
\end{pmatrix}
\begin{pmatrix}
1& -\frac{r_2(k)\delta_+^{2}(k,\xi)}{1-r_1(k)r_2(k)}e^{-2it\theta}\\
0& 1\\
\end{pmatrix},& k\in(-\infty,-\xi),
\\
\begin{pmatrix}
1& -r_2(k)\delta^2(k,\xi)e^{-2it\theta}\\
0& 1\\
\end{pmatrix}
\begin{pmatrix}
1& 0\\
r_1(k)\delta^{-2}(k,\xi)e^{2it\theta}& 1\\
\end{pmatrix},& k\in(-\xi,\infty)\setminus\{0\},
\end{cases}
\end{equation}
 the residue conditions
\begin{subequations}
\begin{align}
&\underset{k=p_j}{\operatorname{Res}}\tilde M^{(1)}(x,t,k)=
\frac{\eta_j}{\dot a_1(p_j)\delta^{2}(p_j,\xi)}
e^{2ip_jx+4ip_j^2t}\tilde M^{(2)}(x,t,p_j),&& j=\overline{1,n},\\
&\underset{k=-\overline{p}_j}{\operatorname{Res}}\tilde M^{(1)}(x,t,k)=
\frac{1}{\bar\eta_j\dot a_1(-\overline{p}_j)
\delta^{2}(-\overline{p}_j,\xi)}
e^{-2i\overline{p}_jx+4i\overline{p}_j^2t}\tilde M^{(2)}(x,t,-\overline{p}_j),
&& j=\overline{1,n},
\end{align}
\end{subequations}
and the pseudo-residue conditions at $k=0$:
\begin{subequations}\label{14}
\begin{align}
\label{14.2}
&\tilde{M}(x,t,k)=
\begin{pmatrix}
-\frac{4v_1(x,t)}{A^2a_2(0)\delta(0,\xi)}& \delta(0,\xi)\overline{v_2}(-x,t)\\
-\frac{4v_2(x,t)}{A^2a_2(0)\delta(0,\xi)}& -\delta(0,\xi)\overline{v_1}(-x,t)
\end{pmatrix}
(I+O(k))
\begin{pmatrix}
k& 0\\
0& \frac{1}{k}
\end{pmatrix},&& k\to 0, \,\, k\in\mathbb{C}^{+},\\
\label{14.3}
&\tilde{M}(x,t,k)=\frac{2i}{A}
\begin{pmatrix}
\frac{\overline{v_2}(-x,t)}{\delta(0,\xi)}& -\delta(0,\xi)\frac{v_1(x,t)}{a_2(0)}\\
\frac{-\overline{v_1}(-x,t)}{\delta(0,\xi)}& -\delta(0,\xi)\frac{v_2(x,t)}{a_2(0)}
\end{pmatrix}
+O(k),&& k\to 0, \,\, k\in\mathbb{C}^{-}.
\end{align}
\end{subequations}
Moreover,  $\tilde{M}(x,t,k)$ is, in general, singular at 
$k =-\xi$:
\begin{equation}\label{p-res-2}
\tilde{M}_{\pm}(x,t,k)=
\left(\tilde{M}_{\pm}(x,t)+O(k+\xi)\right)(k+\xi)^{\Im\nu(-\xi)\sigma_3},\quad
k\to-\xi,
\end{equation}
where $\det\tilde{M}_{\pm}(x,t)=1$ for all $x,t$.

Notice that conditions (\ref{tildeM})--(\ref{p-res-2}) determine a RH problem
whose solution is unique, if exists, for any value of $\Im\nu(-\xi)$. 

\subsection{The RH problem deformations}\label{RHdeform}
The triangular factorizations (\ref{as4}) 
suggest deforming the contour for the 
 Riemann-Hilbert problem to a cross  centered at $k=-\xi$ (see Figure \ref{F1}), so that the (transformed)
jump matrix  converges (as $t\to\infty$)  to the identity matrix 
exponentially fast away from a  neighborhood of $k=-\xi$. 
In general, in order to be able to do this, we have to use analytic 
approximations of the 
 reflection coefficients $r_j(k)$, $j=1,2$ outside the real axis.
In order to   avoid  technical complications related to such approximations
and to keep transparent the 
realization of the main ideas of the asymptotic analysis, 
we assume in what follows that $r_j(k)$, $j=1,2$ are analytic 
at least in a band containing the real axis. This assumption holds, for example, if the initial value $q_0(x)$ is a local (with a finite support) perturbation of the background step function.

Adopting the notations $\hat\Omega_j$, $j=0,\dots,4$ for the sectors as in Figure \ref{F1} (notice that the points $\{p_j\}_1^n$ and  $\{-\bar{p}_j\}_1^n$ are located in $\hat{\Omega}_0$), we 
 define $\hat{M}(x,t,k)$ as follows:
\begin{figure}[h]
\centering{\includegraphics[width=0.99\linewidth]{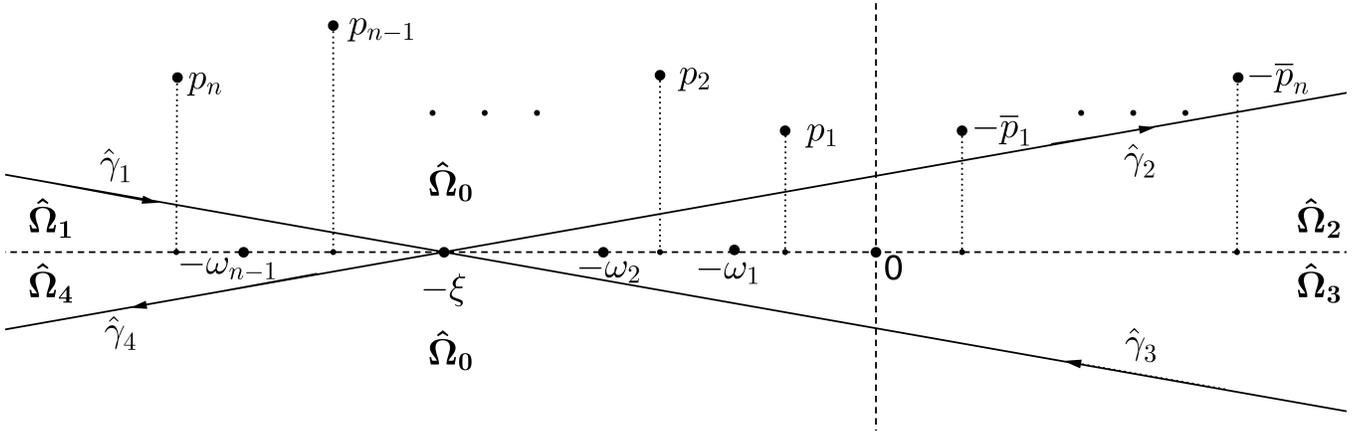}}
\caption{The domains $\hat\Omega_j$, $j=0,\dots,4$  and the contour 
$\hat\Gamma=\hat\gamma_1\cup...\cup\hat\gamma_4$}
\label{F1}
\end{figure}
\begin{equation}
\hat{M}(x,t,k)=
\begin{cases}
\tilde{M}(x,t,k),& k\in\hat\Omega_0,\\
\tilde{M}(x,t,k)
\begin{pmatrix}
1& \frac{r_2(k)\delta^{2}(k,\xi)}{1-r_1(k)r_2(k)}e^{-2it\theta}\\
0& 1\\
\end{pmatrix}
,& k\in\hat\Omega_1,
\\
\tilde{M}(x,t,k)
\begin{pmatrix}
1& 0\\
-r_1(k)\delta^{-2}(k,\xi)e^{2it\theta}& 1\\
\end{pmatrix}
,& k\in\hat\Omega_2,
\\
\tilde{M}(x,t,k)
\begin{pmatrix}
1& -r_2(k)\delta^2(k,\xi)e^{-2it\theta}\\
0& 1\\
\end{pmatrix}
,& k\in\hat\Omega_3,
\\
\tilde{M}(x,t,k)
\begin{pmatrix}
1& 0\\
\frac{r_1(k)\delta^{-2}(k,\xi)}{1-r_1(k)r_2(k)}e^{2it\theta}& 1\\
\end{pmatrix}
,& k\in\hat\Omega_4.
\end{cases}
\end{equation}
Then $\hat M(x,t,k)$ satisfies the jump (across $\hat\Gamma$) and norming
\begin{subequations}\label{RHhat}
\begin{align}
&\hat{M}_+(x,t,k)=\hat{M}_-(x,t,k)\hat{J}(x,t,k),&&k\in\hat\Gamma,\\
&\hat{M}(x,t,k)\rightarrow I,&&k\rightarrow\infty,
\end{align}
\end{subequations}
with the jump matrix 
\begin{equation}
\label{J-hat}
\hat{J}(x,t,k)=
\begin{cases}
\begin{pmatrix}
1& \frac{-r_2(k)\delta^{2}(k,\xi)}{1-r_1(k)r_2(k)}e^{-2it\theta}\\
0& 1\\
\end{pmatrix}
,& k\in\hat\gamma_1,
\\
\begin{pmatrix}
1& 0\\
r_1(k)\delta^{-2}(k,\xi)e^{2it\theta}& 1\\
\end{pmatrix}
,& k\in\hat\gamma_2,
\\
\begin{pmatrix}
1& r_2(k)\delta^2(k,\xi)e^{-2it\theta}\\
0& 1\\
\end{pmatrix}
,& k\in\hat\gamma_3,
\\
\begin{pmatrix}
1& 0\\
\frac{-r_1(k)\delta^{-2}(k,\xi)}{1-r_1(k)r_2(k)}e^{2it\theta}& 1\\
\end{pmatrix}
,& k\in\hat\gamma_4,
\end{cases}
\end{equation}
the residue conditions
\begin{subequations}
\begin{align}
\label{respj}
&\underset{k=p_j}{\operatorname{Res}}\hat M^{(1)}(x,t,k)=f_j(x,t)\hat M^{(2)}(x,t,p_j),\quad j=\overline{1,n},\\
&\underset{k=-\overline{p}_j}{\operatorname{Res}}\hat M^{(1)}(x,t,k)=
\tilde f_j(x,t)\hat M^{(2)}(x,t,-\overline{p}_j),\quad j=\overline{1,n},
\end{align}
\end{subequations}
with
\begin{equation}\label{f-j}
f_j(x,t)=\frac{\eta_je^{2ip_jx+4ip_j^2t}}{\dot a_1(p_j)\delta^{2}(p_j,\xi)},\quad 
\tilde f_j(x,t)=\frac{e^{-2i\overline{p}_jx+4i\overline{p}_j^2t}}
{\bar\eta_j\dot a_1(-\overline{p}_j)
\delta^{2}(-\overline{p}_j,\xi)},
\end{equation}
 the residue condition at $k=0$:
\begin{equation}\label{zc}
\underset{k=0}{\operatorname{Res}}\hat M^{(2)}(x,t,k)=c_0(\xi)
\hat M^{(1)}(x,t,0),
\end{equation}
with $c_0(\xi)=\frac{A\delta^2(0,\xi)}{2i}$,
and the (singular) behavior at $k=-\xi$:
\begin{equation}\label{omegahat}
\hat{M}(x,t,k)=
\left(\hat{M}_{-\xi}(x,t)+O(k+\xi)\right)
(k+\xi)^{\Im\nu(-\xi)\sigma_3},\quad
k\to-\xi,
\end{equation}
where $\hat{M}_{-\xi}(x,t)$ is  some 
matrix function with $\det\hat{M}_{-\xi}(x,t)=1$ for all $x$ and $t$.

Notice that the pseudo-residue conditions (\ref{z}) have transformed  into 
(\ref{zc}), the latter having the form of  a conventional residue condition.

\begin{proposition}\label{asRH}
For any fixed $\xi=\frac{x}{4t}$, $\xi>0$ such that 
$\xi\not\in \{\omega_m\}_1^{n-1}\cup\{\Re p_m\}_1^n \cup\{0\}$,
the solution of the Riemann-Hilbert problem (\ref{RHhat})-(\ref{omegahat}) can be approximated (as $t\to\infty$) by 
the solution of a RH problem (denoted by $M^{as}$)  characterized by a single 
residue condition (at $k=0$) and a weak singularity at $k=-\xi$: 
\begin{align}\label{solas}
&q(x,t)=2i\lim_{k\to\infty}kM_{12}^{as}(\xi,t,k)+ O(e^{-Ct}),\quad t\to\infty,\\
\label{sol1as}
&q(-x,t)=-2i\lim_{k\to\infty}k\overline{M_{21}^{as}(\xi,t,k)}+
O(e^{-Ct}),\quad t\to\infty,
\end{align}
with some $C>0$.
Depending on the value of $\xi$, the approximating RH problem 
has one of two forms
(to make the presentation more compact, we adopt the convention
$\prod\limits_{s=m_1}^{m_2}(\cdot)_s=1$, if $m_1>m_2$):
\begin{enumerate}[(i)]
\item for $-\omega_{n-m}<-\xi<\Re p_{n-m}$, $m=\overline{0,n-1}$, $M^{as}$ solves
the RH problem
\begin{subequations}\label{RHas1}
\begin{align}
&M^{as}_+(\xi,t,k)=M^{as}_-(\xi,t,k)J^{as}(\xi,t,k),&& k\in\hat\Gamma,\\
&M^{as}(\xi,t,k)\to I, && k\to\infty,\\
\label{RHas1c}
&\underset{k=0}{\operatorname{Res}}M^{as\,(2)}(\xi,t,k)=
c_{0}^{as}(\xi)M^{as\,(1)}(\xi,t,0),\\
\label{singxi1}
&M^{as}(\xi,t,k)=
\left(M^{as}_{-\xi}(\xi,t)+O(k+\xi)\right)
(k+\xi)^{(\Im\nu(-\xi)-m)\sigma_3},&&
k\to-\xi,
\end{align}
\end{subequations}
where
\begin{equation}\label{c0as}
c_{0}^{as}(\xi)=\frac{A\delta^2(0,\xi)}{2i}
\prod\limits_{s=0}^{m-1}\left(\frac{\xi}{p_{n-s}}\right)^2
\end{equation}
and
\begin{equation}\label{RHas1j}
J^{as}(\xi,t,k)=
\left(\prod\limits_{s=0}^{m-1}\frac{k+\xi}{k-p_{n-s}}\right)^{\sigma_3}
\hat{J}(x,t,k)
\left(\prod\limits_{s=0}^{m-1}\frac{k+\xi}{k-p_{n-s}}\right)
^{-\sigma_3},\,k\in\hat\Gamma.
\end{equation}

\item for $\Re p_{n-m}<-\xi<-\omega_{n-m-1}$, $m=\overline{0,n-1}$, $M^{as}$ solves
the RH problem
\begin{subequations}\label{RHas2}
\begin{align}
&M^{as}_+(\xi,t,k)=M^{as}_-(\xi,t,k)J^{as}(\xi,t,k),&& k\in\hat\Gamma,\\
&M^{as}(\xi,t,k)\to I, && k\to\infty,\\
&\underset{k=0}{\operatorname{Res}}\ M^{as\,(1)}(\xi,t,k)=
c_{0}^{as\#}(\xi)M^{as\,(2)}(\xi,t,0),\\
&M^{as}(x,t,k)=
\left(M^{as}_{-\xi}(x,t)+O(k+\xi)\right)
(k+\xi)^{(\Im\nu(-\xi)-m)\sigma_3},&&
k\to-\xi,
\end{align}
\end{subequations}
where 
\begin{equation}\label{c0as2}
c_0^{as\#}(\xi)=\frac{2ip_{n-m}^2}{A\delta^2(0,\xi)}
\prod\limits_{s=0}^{m-1}\left(\frac{p_{n-s}}{\xi}\right)^2
\end{equation}
and
\begin{equation}\label{J-ii}
J^{as}(\xi,t,k)=
\left(d(k)
\prod\limits_{s=0}^{m-1}\frac{k+\xi}{k-p_{n-s}}
\right)^{\sigma_3}
\hat{J}(x,t,k)
\left(d(k)
\prod\limits_{s=0}^{m-1}\frac{k+\xi}{k-p_{n-s}}\right)^{-\sigma_3},
\,k\in\hat\Gamma,
\end{equation}
with $d(k)=\frac{k}{k-p_{n-m}}$.
\end{enumerate}
\end{proposition}
\begin{proof}
\textit{(i)} Consider $\xi$ such that $-\omega_{n-m}<-\xi<\Re p_{n-m}$, $m=\overline{0,n-1}$. In this case, the RH problem for 
$\hat M(x,t,k)$ involves $m$  residue conditions, at  $k=p_{n-s}$, $s=\overline{0,m-1}$, with exponentially growing factors, the other having exponentially decaying factors,
see (\ref{respj}). Then, introducing $\check M(x,t,k)$ by
\begin{equation}\label{singul}
\check M(x,t,k):=\hat M(x,t,k) 
\left(
\prod\limits_{s=0}^{m-1}\frac{k+\xi}{k-p_{n-s}}
\right)^{-\sigma_3},\quad k\in\mathbb{C},
\end{equation}
we have (see (\ref{RHas1c})) that  $\check M(x,t,k)$ satisfies a RH problem 
with all residue conditions but one (at $k=0$) having exponentially decaying
 factors. Moreover, $\check M(x,t,k)$
satisfies the jump condition with the jump matrix given by (\ref{RHas1j})
and  has a weak singularity of type (\ref{singxi1})
(in view of (\ref{3.8}), $\Im \nu(-\xi)-m \in (-\frac{1}{2}, \frac{1}{2})$).
Ignoring the residue conditions with decaying factors, we arrive 
at the RH problem (\ref{RHas1}).

\textit{(ii)} Now consider $\Re p_{n-m}<-\xi<-\omega_{n-m-1}$, $m=\overline{0,n-1}$.
In this case, the RH problem for  $\hat M(x,t,k)$ involves $m+1$ 
residue conditions, at  $k=p_{n-s}$, $s=\overline{0,m}$, with exponentially growing factors. Applying the transformation (\ref{singul})  and ignoring the exponentially decaying residue conditions, we arrive
at   the  RH problem with two residue conditions, one of them having 
an exponentially growing factor:
$k=p_{n-m}$ (see (\ref{f-j})):
\begin{subequations}\label{RHaux}
\begin{align}
&\tilde M^{as}_+(\xi,t,k)=\tilde M^{as}_-(\xi,t,k)\tilde J^{as}(\xi,t,k),&& k\in\hat\Gamma,\\
&\tilde M^{as}(\xi,t,k)\to I, && k\to\infty,\\
\label{RHauxc}
&\underset{k=p_{n-m}}{\operatorname{Res}}\tilde M^{as\,(1)}(\xi,t,k)=
f(x,t)\tilde M^{as\,(2)}(\xi,t,p_{n-m}),\\
\label{RHauxd}
&\underset{k=0}{\operatorname{Res}}\tilde M^{as\,(2)}(\xi,t,k)=
c_{0}^{as}(\xi)\tilde M^{as\,(1)}(\xi,t,0),\\
&\tilde M^{as}(x,t,k)=\left(\tilde M^{as}_{-\xi}(x,t)+O(k+\xi)\right)
(k+\xi)^{(\Im\nu(-\xi)-m)\sigma_3},\,
k\to-\xi,
\end{align}
\end{subequations}
where $f(x,t)=f_{n-m}(x,t)\prod\limits_{s=0}^{m-1}
\left(\frac{p_{n-m}-p_{n-s}}{p_{n-m}+\xi}\right)^2$, 
$c_0^{as}(\xi)$ is given by (\ref{c0as}), and
$$
\tilde J^{as}(\xi,t,k)=
\left(
\prod\limits_{s=0}^{m-1}\frac{k+\xi}{k-p_{n-s}}
\right)^{\sigma_3}
\hat{J}(x,t,k)
\left(
\prod\limits_{s=0}^{m-1}\frac{k+\xi}{k-p_{n-s}}
\right)^{-\sigma_3},\,k\in\hat \Gamma.
$$

The latter problem has two residue conditions for the different columns, where one of them is exponentially growing and the other one is bounded. 
Problems of this type can be transformed (see e.g. \cite{DKKZ, RSfs}) in such a way that the exponentially growing conditions ((\ref{RHauxc}), in our case) 
transform to exponentially decaying. 
Indeed, the problem (\ref{RHaux}) with residue conditions can be transformed into 
a regular problem (for $\hat M^{as}$)  having additional parts of the contour in the form of 
small circles, $S_0$ and $S_{p_{n-m}}$,  surrounding respectively $k=0$ and
$k=p_{n-m}$, and the enhanced jump conditions:
\begin{subequations}\label{RHaux2}
	\begin{align}
	&\hat M^{as}_+(x,t,k)=\hat M^{as}_-(x,t,k)\hat J^{as}(x,t,k),&& k\in\hat\Gamma\cup S_0\cup S_{p_{n-m}},\\
	&\hat M^{as}(x,t,k)\to I, && k\to\infty,
	\end{align}
\end{subequations}
where 
\begin{equation}
\hat J^{as}(x,t,k)=
\begin{cases}
\tilde J^{as}(x,t,k),& k\in\hat\Gamma,\\
\begin{pmatrix}
1& -\frac{c_0^{as}(\xi)}{k}\\
0& 1
\end{pmatrix}, & k\in S_0,\\
\begin{pmatrix}
1& 0\\
-\frac{f(x,t)}{k-p_{n-m}}& 1
\end{pmatrix},& k\in S_{p_{n-m}}.
\end{cases}
\end{equation}
Next, introducing $\hat M^{as\#}$ by
\begin{equation}
\label{35}
\hat M^{as\#}(x,t,k)=
\begin{cases}
\hat{M}^{as}(x,t,k)N(\xi,k)
d^{-\sigma_3}(k),
& k\mbox{ inside } S_0,\\
\hat{M}^{as}(x,t,k)Q(x,t,k)
d^{-\sigma_3}(k),
& k\mbox{ inside } S_{p_{n-m}},\\
\hat{M}^{as}(x,t,k)
d^{-\sigma_3}(k),
& \mbox{ otherwise },
\end{cases}
\end{equation}
where $d(k)=\frac{k}{k-p_{n-m}}$,
$
N(\xi,k)=
\begin{pmatrix}
0& \frac{c_0^{as}(\xi)}{k}\\
-\frac{k}{c_0^{as}(\xi)}& 1
\end{pmatrix}
$, and
$
Q(x,t,k)=
\begin{pmatrix}
1& -\frac{k-p_{n-m}}{f(x,t)}\\
\frac{f(x,t)}{k-p_{n-m}}& 0
\end{pmatrix},
$
straightforward  calculations show that the jump matrices for $\hat M^{as\#}$ 
across $S_{p_{n-m}}$ are  exponentially decaying (to the identity matrix),
whereas the jump across $\hat\Gamma$ takes the form (\ref{J-ii}).
\end{proof}

Neglecting the jump conditions in (\ref{RHas1}) and (\ref{RHas2})
(recall that due to the signature table, the jump matrices decay, as $t\to\infty$,
 to the identity matrix exponentially fast outside any vicinity of 
$k=-\xi$), the RH problems (\ref{RHas1}) and (\ref{RHas2}) reduce, as $t\to\infty$,
 to algebraic equations that 
can be solved explicitly:
\begin{equation}\label{MRHasymp}
M^{as}(\xi,t,k)\simeq
\begin{cases}
\begin{pmatrix}
1 & \frac{c_0^{as}(\xi)}{k}\\
0 & 1
\end{pmatrix}
& \mbox{for RH (\ref{RHas1})},\\
\begin{pmatrix}
1 & 0\\
\frac{c_0^{as\#}(\xi)}{k} & 1
\end{pmatrix},
& \mbox{for RH (\ref{RHas2})},
\end{cases}
\end{equation}
where $c_0^{as}(\xi)$ and $c_0^{as\#}(\xi)$ are given by (\ref{c0as}) and (\ref{c0as2}) respectively. Substituting (\ref{MRHasymp}) into  (\ref{solas}) and (\ref{sol1as}),
the (rough) asymptotics (\ref{asq1}) in Theorem \ref{cor1} follow.

\begin{remark}
\label{rem4}
For the focusing NNLS equation \cite{RSfs}, the pure step initial function
with   $R=0$ (i.e., $q_0(x)=q_{0,A}(x)$) satisfies assumptions analogous 
to Assumptions (a)-(c) and thus 
this case is covered by the corresponding asymptotic formulas. 
In contrast with this, for the defocusing NNLS equation,
 $R=0$  is one of the bifurcation values of $R$, see Remark \ref{cdnls}.
\end{remark}

In order to rigorously justify the asymptotic formulas (\ref{asq1}), 
we adapt the nonlinear steepest descent method \cite{DIZ,DZ},
which also allows us 
to make the
asymptotics presented in (\ref{asq1}) more precise.

\begin{theorem}\label{th1}
Consider the initial value problem  (\ref{1}), (\ref{2}). 
Assume that (i) the initial value $q_0(x)$ converges to its boundary values fast enough, (ii) the associated spectral functions $a_j(k)$, $j=1,2$ satisfy Assumptions (a)-(c), and (iii) the spectral functions $r_j(k)$, $j=1,2$
can be analytically continued from the real axis into a band along it.
Assuming also that the solution $q(x,t)$ of (\ref{1}), (\ref{2}) exists, it has the following long-time asymptotics along the rays $\frac{x}{4t}=\xi$:
\begin{description}
\item{\textbf{(i)}} for $-\omega_{n-m}<-\xi<\Re p_{n-m}$, $m=\overline{0,n-1}$,
there are  three types of asymptotics depending on the value of $\Im\nu(-\xi)$:
\begin{enumerate}[1)]
\item if $\Im\nu(-\xi)\in\left(m-\frac{1}{2}, m-\frac{1}{6}\right]$, then
\begin{equation*}
q(x,t)=A\delta^2(0,\xi)
\prod\limits_{s=0}^{m-1}\left(\frac{\xi}{p_{n-s}}\right)^2
+t^{-\frac{1}{2}-\Im\nu(-\xi)+m}\alpha_1(\xi)\exp\{-4it\xi^2+i\Re\nu(-\xi)\ln t\}
+ R_1(\xi,t).
\end{equation*}
\item if $\Im\nu(-\xi)\in\left(m-\frac{1}{6}, m+\frac{1}{6}\right)$, then
\begin{align}
\nonumber
q(x,t)=&A\delta^2(0,\xi)
\prod\limits_{s=0}^{m-1}\left(\frac{\xi}{p_{n-s}}\right)^2
+t^{-\frac{1}{2}-\Im\nu(-\xi)+m}\alpha_1(\xi)\exp\{-4it\xi^2+i\Re\nu(-\xi)
\ln t\}\\
\nonumber
&+t^{-\frac{1}{2}+\Im\nu(-\xi)-m}\alpha_2(\xi)\exp\{4it\xi^2-i\Re\nu(-\xi)
\ln t\}+R_3(\xi,t).
\end{align}
\item if $\Im\nu(-\xi)\in\left[m+\frac{1}{6}, m+\frac{1}{2}\right)$, then
\begin{equation*}
q(x,t)=A\delta^2(0,\xi)
\prod\limits_{s=0}^{m-1}\left(\frac{\xi}{p_{n-s}}\right)^2
+t^{-\frac{1}{2}+\Im\nu(-\xi)-m}\alpha_2(\xi)\exp\{4it\xi^2-i\Re\nu(-\xi)
\ln t\}+R_2(\xi,t).
\end{equation*}
\end{enumerate}
\item{\textbf{(ii)}} for $-\Re p_{n-m}<-\xi<\omega_{n-m}$, $m=\overline{0,n-1}$:
\begin{equation*}
q(x,t)=t^{-\frac{1}{2}-\Im\nu(\xi)+m}\alpha_3(\xi)
\exp\{4it\xi^2-i\Re\nu(\xi)\ln t\}
+ R_2(-\xi,t)
\end{equation*}
\item{\textbf{(iii)}} for $\Re p_{n-m}<-\xi<-\omega_{n-m-1}$, $m=\overline{0,n-1}$:
\begin{equation*}
q(x,t)=t^{-\frac{1}{2}+\Im\nu(-\xi)-m}\alpha_4(\xi)
\exp\{4it\xi^2-i\Re\nu(-\xi)\ln t\}
+ R_2(\xi,t)
\end{equation*}
\item{\textbf{(iv)}} for $\omega_{n-m-1}<-\xi<-\Re p_{n-m}$, $m=\overline{0,n-1}$,
there are  three types of asymptotics  depending on the value of $\Im\nu(\xi)$:
\begin{enumerate}[1)]
\item if $\Im\nu(\xi)\in\left(m-\frac{1}{2}, m-\frac{1}{6}\right]$, then
\begin{equation*}
q(x,t)=\frac{-4\overline{p}_{n-m}^2}{A\overline{\delta^2}(0,-\xi)}
\prod\limits_{s=0}^{m-1}\left(\frac{\overline{p}_{n-s}}
{\xi}\right)^2
+t^{-\frac{1}{2}-\Im\nu(\xi)+m}\alpha_5(\xi)\exp\{4it\xi^2-i\Re\nu(\xi)
\ln t\}+R_1(-\xi,t),
\end{equation*}
\item if $\Im\nu(\xi)\in\left(m-\frac{1}{6}, m+\frac{1}{6}\right)$, then
\begin{align}
\nonumber
q(x,t)=&\frac{-4\overline{p}_{n-m}^2}{A\overline{\delta^2}(0,-\xi)}
\prod\limits_{s=0}^{m-1}\left(\frac{\overline{p}_{n-s}}
{\xi}\right)^2
+t^{-\frac{1}{2}-\Im\nu(\xi)+m}\alpha_5(\xi)\exp\{4it\xi^2-i\Re\nu(\xi)
\ln t\}\\
\nonumber
&+t^{-\frac{1}{2}+\Im\nu(\xi)-m}\alpha_6(\xi)\exp\{-4it\xi^2+i\Re\nu(\xi)
\ln t\}+R_3(-\xi,t).
\end{align}
\item if $\Im\nu(
\xi)\in\left[m+\frac{1}{6}, m+\frac{1}{2}\right)$, then
\begin{equation*}
q(x,t)=\frac{-4\overline{p}_{n-m}^2}{A\overline{\delta^2}(0,-\xi)}
\prod\limits_{s=0}^{m-1}\left(\frac{\overline{p}_{n-s}}
{\xi}\right)^2
+t^{-\frac{1}{2}+\Im\nu(\xi)-m}\alpha_6(\xi)\exp\{-4it\xi^2+i\Re\nu(\xi)\ln t\}
+ R_2(-\xi,t).
\end{equation*}
\end{enumerate}
\end{description}
Here
\begin{equation}
\delta(k,\xi)=
(k+\xi)^{i\nu(-\xi)}e^{\chi(k,\xi)},
\end{equation}
and
\begin{equation}
\nu(-\xi)=-\frac{1}{2\pi}\ln|1+r_1(-\xi)r_2(-\xi)|
-\frac{i}{2\pi}
\int_{-\infty}^{-\xi}\,d\arg(1-r_1(\zeta)r_2(\zeta)),
\end{equation}
with
\begin{equation}
\chi(k,\xi)=-\frac{1}{2\pi i}\int_{-\infty}^{-\xi}\ln(k-\zeta)
d_{\zeta}(1-r_1(\zeta)r_2(\zeta)).
\end{equation}
The modulating functions $\alpha_j(\xi)$, $j=\overline{1,6}$ are as follows:
$$
\alpha_1(\xi)=
\frac{\sqrt{\pi}(c_0^{as}(\xi))^2\prod\limits_{s=0}^{m-1}(\xi+p_{n-s})^2}
{\xi^2 r_2(-\xi)\Gamma(i\nu(-\xi)+m)}
\exp\left\{-\frac{\pi}{2}(\nu(-\xi)-im)+\frac{3\pi i}{4}
-2\chi(-\xi,\xi)+3(i\nu(-\xi)+m)\ln 2\right\},
$$
$$
\alpha_2(\xi)=
\frac{\sqrt{\pi}\prod\limits_{s=0}^{m-1}(\xi+p_{n-s})^{-2}}
{r_1(-\xi)\Gamma(-i\nu(-\xi)-m)}
\exp\left\{-\frac{\pi}{2}(\nu(-\xi)-im)
+\frac{\pi i}{4}+2\chi(-\xi,\xi)-3(i\nu(-\xi)+m)\ln 2\right\},
$$
$$
\alpha_3(\xi)=
\frac{\sqrt{\pi}\prod\limits_{s=0}^{m-1}(\overline{p}_{n-s}-\xi)^{2}}
{\overline{r_2}(\xi)\Gamma(-i\overline{\nu(\xi)}+m)}
\exp\left\{-\frac{\pi}{2}(\overline{\nu(\xi)}+im)+\frac{\pi i}{4}
-2\overline{\chi(\xi,\xi)}-3(i\overline{\nu(\xi)}-m)\ln 2\right\},
$$
$$
\alpha_4(\xi)=
\frac{\sqrt{\pi}\xi^2\prod\limits_{s=0}^{m}(\xi+p_{n-s})^{-2}}
{r_1(-\xi)\Gamma(-i\nu(-\xi)-m)}
\exp\left\{-\frac{\pi}{2}(\nu(-\xi)-im)+\frac{\pi i}{4}
+2\chi(-\xi,\xi)-3(i\nu(-\xi)+m)\ln 2\right\},
$$
$$
\alpha_5(\xi)=
\frac{\sqrt{\pi}\prod\limits_{s=0}^{m}(\overline{p}_{n-s}-\xi)^{2}}
{\xi^2\overline{r_2}(\xi)\Gamma(-i\overline{\nu(\xi)}+m)}
\exp\left\{-\frac{\pi}{2}(\overline{\nu(\xi)}+im)+\frac{\pi i}{4}
-2\overline{\chi(\xi,\xi)}-3(i\overline{\nu(\xi)}-m)
\ln 2\right\},
$$
$$
\alpha_6(\xi)=
\frac{\sqrt{\pi}\left(\overline{c_0^{as\#}(-\xi)}\right)^2
\prod\limits_{s=0}^{m}(\overline{p}_{n-s}-\xi)^{-2}}
{\overline{r_1}(\xi)\Gamma(i\overline{\nu(\xi)}-m)}
\exp\left\{-\frac{\pi}{2}(\overline{\nu(\xi)}+im)+\frac{3\pi i}{4}
+2\overline{\chi(\xi,\xi)}+3(i\overline{\nu(\xi)}-m)
\ln 2\right\},
$$
where 
$$
c_{0}^{as}(\xi)=\frac{A\delta^2(0,\xi)}{2i}
\prod\limits_{s=0}^{m-1}\left(\frac{\xi}{p_{n-s}}\right)^2,\quad
c_0^{as\#}(\xi)=\frac{2ip_{n-m}^2}{A\delta^2(0,\xi)}
\prod\limits_{s=0}^{m-1}\left(\frac{p_{n-s}}{\xi}\right)^2.
$$
Finally, the remainders $R_j(\xi,t)$, $j=\overline{1,3}$ are estimated as follows:
\begin{equation}
\label{R1}
R_1(\xi,t)=
\begin{cases}
O\left(t^{-1}\right),& \Im\nu(-\xi)>m,\\
O\left(t^{-1}\ln t\right),&\Im\nu(-\xi)=m,\\
O\left(t^{-1+2|\Im\nu(-\xi)-m|}\right),&\Im\nu(-\xi)<m,
\end{cases}
\end{equation}
\begin{equation}
\label{R2}
R_2(\xi,t)=
\begin{cases}
O\left(t^{-1+2|\Im\nu(-\xi)-m|}\right),& \Im\nu(-\xi)>m,\\
O\left(t^{-1}\ln t\right),&\Im\nu(-\xi)=m,\\
O\left(t^{-1}\right),&\Im\nu(-\xi)<m,
\end{cases}
\end{equation}
and
\begin{equation*}
R_3(\xi,t)=R_1(\xi,t) + R_2(\xi,t)=
\begin{cases}
O\left(t^{-1+2|\Im\nu(-\xi)-m|}\right),&\Im\nu(-\xi)\not=m,\\
O\left(t^{-1}\ln t\right),&\Im\nu(-\xi)=m.
\end{cases}
\end{equation*}
\end{theorem}
\begin{proof}

The implementation of the nonlinear steepest descent method to 
the Riemann-Hilbert problems (\ref{RHas1}) and (\ref{RHas2}) is similar, in many aspects, to that in the case of the focusing NNLS \cite{RSfs}. Therefore, in what follows 
we discuss the peculiarities of realization of the method only and refer the reader to \cite{RSfs} for details.

We begin with the analysis of the Riemann-Hilbert problem (\ref{RHas1}) (the analysis of (\ref{RHas2}) is similar). First, we make the following transformation:
\begin{equation}
\nonumber
\check M^{as}(x,t,k)=
\begin{cases}
M^{as}(x,t,k)
\begin{pmatrix}
1& -\frac{c_0^{as}}{k}\\
0& 1\\
\end{pmatrix},
& k\mbox{ inside } S_{0},\\
M^{as}(x,t,k),
& \mbox{ otherwise},
\end{cases}
\end{equation}
where $S_0=\{k: |k|<\varepsilon\}$ with $\varepsilon>0$ small enough.
Then $\check M^{as}(x,t,k)$ solves the Riemann-Hilbert problem without residue condition:
\begin{subequations}
\begin{align}
&\check M^{as}_+(x,t,k)=\check M^{as}_-(x,t,k)\check J^{as}(x,t,k),&& k\in\hat\Gamma\cup S_0,\\
&\check M^{as}(x,t,k)\to I, && k\to\infty,
\end{align}
\end{subequations}
with
\begin{equation}
\check J^{as}(x,t,k)=
\begin{cases}
J^{as}(x,t,k),& k\in\hat\Gamma,\\
\begin{pmatrix}
1& -\frac{c_0^{as}}{k}\\
0& 1
\end{pmatrix}, & k\in S_0.
\end{cases}
\end{equation}
Taking into account (\ref{deltasing}), the jump matrix $\check J^{as}$ on $\hat\Gamma$ can be written as follows:

\begin{equation}
\label{Jas-hat}
\hat{J}(x,t,k)=
\begin{cases}
\begin{pmatrix}
1& \frac{-\check r^{as}_2(k)(k+\xi)^{2i\check\nu(-\xi)}}{1-\check r^{as}_1(k)\check r^{as}_2(k)}e^{-2it\theta+2\chi(k,\xi)}\\
0& 1\\
\end{pmatrix}
,& k\in\hat\gamma_1,
\\
\begin{pmatrix}
1& 0\\
\check r^{as}_1(k)(k+\xi)^{-2i\check\nu(-\xi)}e^{2it\theta-2\chi(k,\xi)}& 1\\
\end{pmatrix}
,& k\in\hat\gamma_2,
\\
\begin{pmatrix}
1& \check r^{as}_2(k)(k+\xi)^{2i\check\nu(-\xi)}e^{-2it\theta+2\chi(k,\xi)}\\
0& 1\\
\end{pmatrix}
,& k\in\hat\gamma_3,
\\
\begin{pmatrix}
1& 0\\
\frac{-\check r^{as}_1(k)(k+\xi)^{-2i\check\nu(-\xi)}}{1-r^{as}_1(k)r^{as}_2(k)}
e^{2it\theta-2\chi(k,\xi)}& 1\\
\end{pmatrix}
,& k\in\hat\gamma_4,
\end{cases}
\end{equation}
where 
\begin{subequations}
\begin{align}
\label{r_1^{as}}
&\check r_1^{as}(k)=r_1(k)\prod\limits_{s=0}^{m-1}(k-p_{n-s})^2,\\
\label{r_2^{as}}
&\check r_2^{as}(k)=r_2(k)\prod\limits_{s=0}^{m-1}(k-p_{n-s})^{-2},\\
&i\check\nu(-\xi)=i\nu(-\xi)+m.
\end{align}
\end{subequations}
Now we  introduce the local parametrix $\check m^{as}_0(x,t,k)$ using  arguments similar to those in the case of the local  nonlinear Schr\"odinger equation (see e.g. \cite{DIZ, FIKK, I1, Len15}):
\begin{equation}\label{m0-R}
\check m^{as}_0(x,t,k)=\Delta(\xi,t)m^{\Gamma}(\xi,z(k)) \Delta^{-1}(\xi,t),
\end{equation}
where $z(k)$ is the rescaled variable defined by
\begin{equation}\label{k-z}
k=\frac{z}{\sqrt{8t}}-\xi,
\end{equation}
\begin{equation}\label{Delta}
\Delta(\xi,t) = e^{(2 i t \xi^2 + \chi(-\xi,\xi))\sigma_3}
(8t)^{\frac{-i\check\nu(-\xi)}{2}\sigma_3},
\end{equation}
 $m^{\Gamma}(\xi,z)$ is determined by 
\begin{equation}\label{m-g-0}
m^{\Gamma}(\xi,z) = m_0(\xi,z) D^{-1}_{j}(\xi,z),\qquad z\in\Omega_j,\,\,j=\overline{0,4},
\end{equation}
see Figure \ref{mod2},
where 
$$
D_0(\xi,z)=e^{-i\frac{z^2}{4}\sigma_3}z^{i\check \nu(-\xi)\sigma_3},
$$
\begin{equation}
\nonumber
\begin{matrix}
D_1(\xi,z)=D_0(\xi,z)
\begin{pmatrix}
1& \frac{\check r^{as}_2(-\xi)}{1+\check r^{as}_1(-\xi)\check r^{as}_2(-\xi)}\\
0& 1\\
\end{pmatrix},
&&
D_2(\xi,z)=D_0(\xi,z)
\begin{pmatrix}
1& 0\\
\check r^{as}_1(-\xi)& 1\\
\end{pmatrix},\\
D_3(\xi,z)=D_0(\xi,z)
\begin{pmatrix}
1& -\check r^{as}_2(-\xi)\\
0& 1\\
\end{pmatrix},
&&
D_4(\xi,z)=D_0(\xi,z)
\begin{pmatrix}
1& 0\\
\frac{-\check r^{as}_1(-\xi)}{1+\check r^{as}_1(-\xi)\check r^{as}_2(-\xi)}& 1\\
\end{pmatrix},
\end{matrix}
\end{equation}
and $m_0(\xi,z)$ is the solution of the following RH problem with a constant jump matrix:
\begin{subequations}\label{as8}
	\begin{align}
	&m_{0+}(\xi,z)=m_{0-}(\xi,z)j_0(\xi),&& z\in\mathbb{R},\\
	&m_0(\xi,z)= \left(I+O(1/z)\right)
	e^{-i\frac{z^2}{4}\sigma_3}z^{i\check \nu(-\xi)\sigma_3},&& z\rightarrow\infty,
	\end{align}
\end{subequations}
where  
\begin{equation}\label{j0}
j_0(\xi)=
\begin{pmatrix}
1+ \check r^{as}_1(-\xi)\check r^{as}_2(-\xi) & \check r^{as}_2(-\xi)\\
\check r^{as}_1(-\xi) & 1
\end{pmatrix}.
\end{equation}

\begin{figure}[h]
	\begin{minipage}[h]{0.99\linewidth}
		\centering{\includegraphics[width=0.5\linewidth]{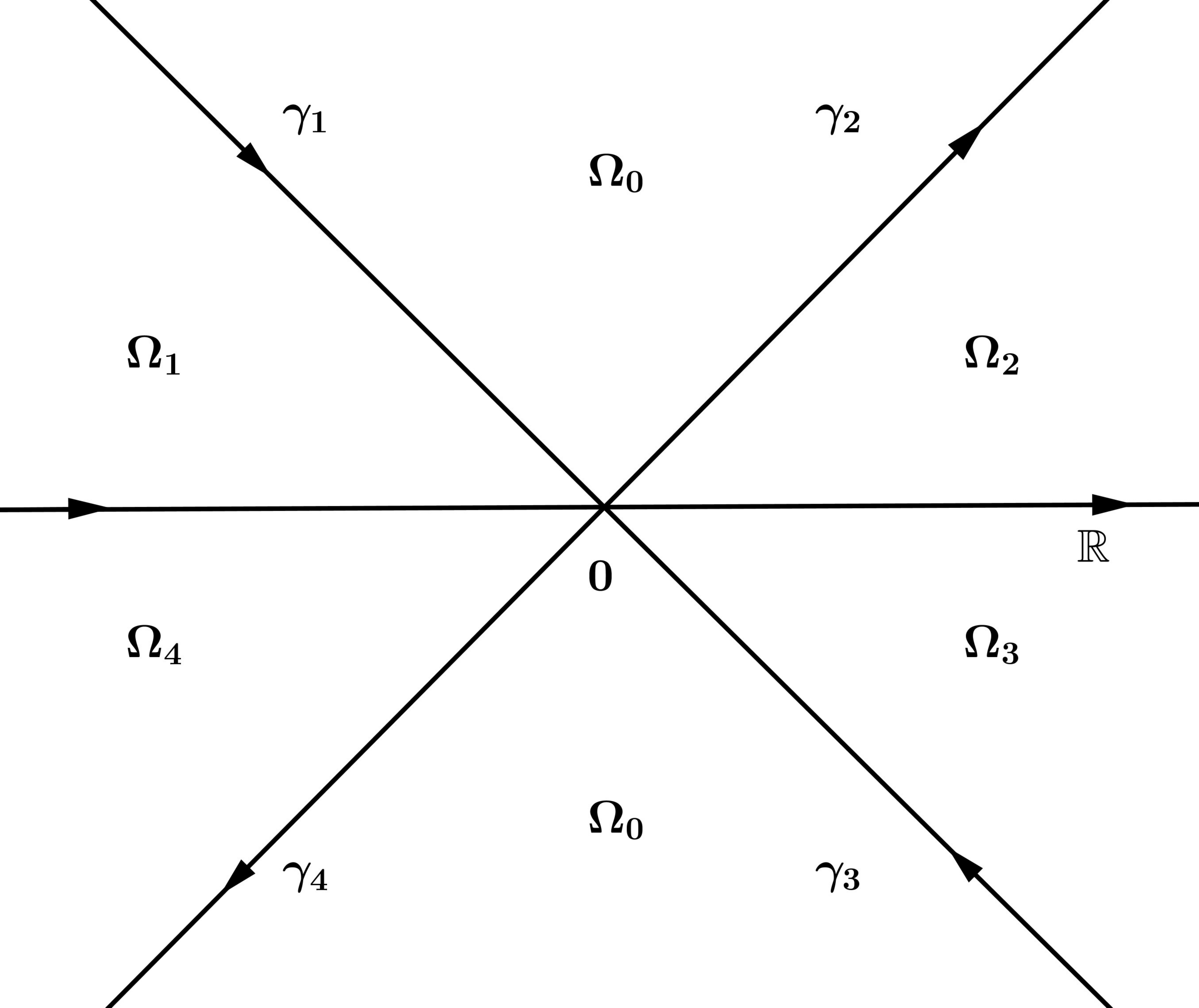}}
		\caption{Contour and domains for $m^{\Gamma}(\xi,z)$ in the $z$-plane  }
		\label{mod2}
	\end{minipage}
\end{figure}

The latter problem can be solved explicitly in terms of the parabolic
cylinder functions \cite{I1}.
For obtaining the long time asymptotics of $q(x,t)$ we need the large-$z$ asymptotics of $m^{\Gamma}(\xi,z)$:
\[
m^{\Gamma}(\xi,z) = I + \frac{i}{z}\begin{pmatrix}
0 & \beta(\xi) \\ -\gamma(\xi) & 0
\end{pmatrix} + O(z^{-2}), \qquad z\to \infty,
\]
where
\begin{subequations}\label{be-ga-R}
\begin{align}
\beta(\xi)=\dfrac{\sqrt{2\pi}e^{-\frac{\pi}{2}\check\nu(-\xi)}e^{-\frac{3\pi i}{4}}}{\check r^{as}_1(-\xi)
\Gamma(-i\check\nu(-\xi))},\\
\gamma(\xi)=\dfrac{\sqrt{2\pi}e^{-\frac{\pi}{2}\check\nu(-\xi)}e^{-\frac{\pi i}{4}}}{\check r^{as}_2(-\xi)\Gamma(i\check\nu(-\xi))}.
\end{align}
\end{subequations}

Similarly to \cite{RSfs}, we define $\breve M^{as}(x,t,k)$ by
\begin{equation}\label{breveM}
\breve M^{as}(x,t,k) = 
\begin{cases}
\check M^{as}(x,t,k)(\check m_0^{as})^{-1}(x,t,k)V(k), & 
k\mbox{ inside } S_{-\xi}, \\
\check M^{as}(x,t,k), & k\mbox{ inside } S_0,\\
\check M^{as}(x,t,k)V(k), & \mbox{ otherwise },
\end{cases}
\end{equation}
where $V(k)=\begin{pmatrix}1& -\frac{c_0^{as}}{k}\\0& 1 \end{pmatrix}$
and $S_{-\xi}$ is a small counterclockwise oriented circle centered at  $k=-\xi$.
Straightforward calculations show that $\breve M^{as}$ solves the following RH problem on $\hat\Gamma_1=\hat\Gamma\cup S_{-\xi}$:
\begin{align}\label{M-breve-RHP}
&\breve M^{as}_+(x,t,k) = \breve M^{as}_-(x,t,k)\breve J^{as}(x,t,k), && k\in\hat\Gamma_1,\\
&\breve M^{as}(x,t,k)\to I, && k\to \infty,
\end{align}
with the jump matrix
\begin{equation}\label{J-breve}
\breve J^{as}(x,t,k) = 
\begin{cases}
V^{-1}(k)\check m_{0-}^{as}(x,t,k) \check J^{as}(x,t,k)
(\check m_{0+}^{as})^{-1}(x,t,k)
V(k), & k\in \hat\Gamma_1, k\mbox{ inside } S_{-\xi},\\
V^{-1}(k)(\check m_{0}^{as})^{-1} (x,t,k)V(k), & k\in S_{-\xi}, \\
V^{-1}(k)\check J^{as}(x,t,k)V(k), & \text{otherwise}.
\end{cases}
\end{equation}
Taking into account (\ref{breveM}), the solution of the original RH problem is given in terms of $\breve M^{as}(x,t,k)$ as follows:
\begin{equation}
q(x,t)=2i\left(c_0^{as}+\lim_{k\to\infty}k\breve M^{as}_{12}(x,t,k)\right),
\end{equation}
and
\begin{equation}
q(-x,t)=-2i\lim_{k\to\infty}k\overline{\breve M^{as}_{21}(x,t,k)}.
\end{equation}

For evaluating the large-$t$ asymptotics of $\breve M^{as}(x,t,k)$ we need the asymptotics of the local parametrix $(\check m_{0}^{as})^{-1} (x,t,k)$:
\begin{equation}\label{m_0}
(\check m_{0}^{as})^{-1} (x,t,k)=\Delta(\xi,t)
(m^{\Gamma})^{-1}(\xi,\sqrt{8t}(k+\xi))\Delta^{-1}(\xi,t)=
I+\frac{B(\xi,t)}{\sqrt{8t}(k+\xi)}+\tilde{r}(\xi,t),
\end{equation}
where the entries of $B(\xi,t)$ are as follows (cf. \cite{RSfs}):
\begin{subequations}\label{B}
\begin{align}
&B_{11}(\xi,t)=B_{22}(\xi,t)=0,\\
&B_{12}(\xi,t)=-i\beta(\xi)e^{4it\xi^2+2\chi(-\xi,\xi)}
(8t)^{-i\check \nu(-\xi)},\\
&B_{21}(\xi,t)=i\gamma(\xi)e^{-4it\xi^2-2\chi(-\xi,\xi)}
(8t)^{i\check \nu(-\xi)},
\end{align}
\end{subequations}
and the remainder is:
\begin{equation}
\tilde r(\xi,t)=
\begin{pmatrix}
O\left(t^{-1-\Im\check\nu(-\xi)}\right)& O\left(t^{-1+\Im\check\nu(-\xi)}\right)\\
O\left(t^{-1-\Im\check\nu(-\xi)}\right)& O\left(t^{-1+\Im\check\nu(-\xi)}\right)
\end{pmatrix},\quad t\to\infty.
\end{equation}

 $\breve M^{as}(x,t,k)$ can be written as 
\begin{equation}\label{singintm}
\breve M^{as}(x,t,k)=I+\frac{1}{2\pi i}\int\limits_{\hat\Gamma_1}\mu(x,t,s)
(\breve J^{as}(x,t,s)-I)\frac{ds}{s-k},
\end{equation}
where $\mu(x,t,k)$ solves the integral equation
\begin{equation}
\mu(x,t,k)=I+\frac{1}{2\pi i}\lim_{\substack{k'\to k\\k'\in-\mbox{side}}}
\int_{\hat\Gamma_1}\frac{\mu(x,t,s)(\breve J^{as}(x,t,s)-I)}{s-k'}\,ds.
\end{equation}
Estimating the rhs of (\ref{singintm}) (cf. \cite{RS}) we conclude that
\begin{align}
\nonumber
\lim_{k\to\infty} k\left(\breve M^{as}(x,t,k) - I\right)&=
-\frac{1}{2\pi i}\int_{S_{-\xi}}V^{-1}(k)((\check m_0^{as})^{-1}(x,t,k)-I)V(k)\,dk+
R(\xi,t)\\
\label{M^R1}
&=B^{as}(\xi,t)+R(\xi,t),
\end{align}
where $R(\xi,t)=
\begin{pmatrix}
R_1(\xi,t)& R_1(\xi,t) + R_2(\xi,t)\\ 
R_1(\xi,t)& R_1(\xi,t) + R_2(\xi,t)
\end{pmatrix}$ and (see (\ref{B}))
\begin{equation}
B^{as}(\xi,t)=\frac{1}{\sqrt{8t}}
\begin{pmatrix}
\frac{c_0^{as}(\xi)}{\xi}B_{21}(\xi,t)&
\frac{(c_0^{as}(\xi))^2}{\xi^2}B_{21}(\xi,t)-B_{12}(\xi,t)\\
-B_{21}(\xi,t)& -\frac{c_0^{as}(\xi)}{\xi}B_{21}(\xi,t)
\end{pmatrix}.
\end{equation}
Replacing  $\check r_j^{as}(-\xi)$, $j=1,2$ and $\check \nu(-\xi)$ by 
$r_j(-\xi)$, $j=1,2$ and $\nu(-\xi)$ respectively, we arrive at asymptotics described by \textbf{(i)} and \textbf{(iii)}.

Returning to the Riemann-Hilbert problem (\ref{RHas2}),
the reflection coefficients $\check r_j^{as}(k)$, $j=1,2$ (see (\ref{r_1^{as}}) and (\ref{r_2^{as}})) have the form
$$
r^{as}_1(k)=r_1(k)d^{-2}(k)
\prod\limits_{s=0}^{m-1}\left(k-p_{n-s}\right)^2,
\quad
r^{as}_2(k)=r_2(k)d^2(k)
\prod\limits_{s=0}^{m-1}\left(k-p_{n-s}\right)^{-2},
$$ where $d(k)=\frac{k}{k-p_{n-m}}$. 
Moreover,
$
V(k)=\begin{pmatrix}1& 0\\-\frac{c_0^{as\#}(\xi)}{k}& 1 \end{pmatrix}
$
in the definition of $\breve M^{as}(x,t,k)$ (see (\ref{breveM})).
Therefore,
\begin{subequations}\label{qsec}
\begin{align}
&q(x,t)=2i\lim_{k\to\infty}k\breve M^{as}_{12}(x,t,k),\\
&q(-x,t)=-2i\left(\overline{c_0^{as\#}(\xi)}+\lim_{k\to\infty}k\overline{\breve M^{as}_{21}(x,t,k)}\right),
\end{align}
\end{subequations}
and  $B^{as}(\xi,t)$ and $R(\xi,t)$ in (\ref{M^R1}) are as follows:
\begin{equation}\label{basces}
B^{as}(\xi,t)=\frac{1}{\sqrt{8t}}
\begin{pmatrix}
-\frac{c_0^{as\#}(\xi)}{\xi}B_{12}(\xi,t)&
-B_{12}(\xi,t)\\
\frac{(c_0^{as\#}(\xi))^2}{\xi^2}B_{12}(\xi,t)-B_{21}(\xi,t)& \frac{c_0^{as\#}(\xi)}{\xi}B_{12}(\xi,t)
\end{pmatrix},
\end{equation}
and  $R(\xi,t)=
\begin{pmatrix}
R_1(\xi,t) + R_2(\xi,t)& R_2(\xi,t)\\
R_1(\xi,t) + R_2(\xi,t)& R_2(\xi,t)
\end{pmatrix}$.
Collecting (\ref{qsec}) and (\ref{basces}) we obtain items \textbf{(ii)} and \textbf{(iv)} of the Theorem.

\end{proof}

\section{Transition regions}

In Theorem \ref{th1} we present the large-time behavior of the solution $q(x,t)$ along the rays $\xi=\frac{x}{4t}=const$ for all $\xi\not\in\{\pm\Re p_m,\pm\omega_{m-1}|m=\overline{1,n}\}$, i.e. for all $\xi\in\mathbb{R}$ except the boundaries of the qualitatively different asymptotic sectors. Since these asymptotic regimes do not match as the slow variable $\xi$ approaches the edges of the sectors, the study of the transition zones between the decaying and constant regimes is a non-trivial task.

We conjecture that there are 3 different types of transition regions in the asymptotics described in Theorem \ref{th1}:
\begin{enumerate}[(1)]
\item zones near the rays $\xi=\pm\Re p_{m}$, $m=\overline{1,n}$,
\item zones as $\xi$ approaches to $\omega_m$, $m=\overline{1,n-1}$, 
\item a zone as $\xi$ approaches to zero.
\end{enumerate}

In the following proposition we describe the transition zones of type (1),
where the transition is described by a solitary kink  propagating
 along the rays $\xi=\pm4\Re p_m$, $m=\overline{1,n}$:
\begin{proposition}
	\label{solit_tr}
	Under assumptions of Theorem \ref{th1}, the solution $q(x,t)$
	of problem (\ref{1}), (\ref{2}) has the following asymptotics along the rays 
	$\xi=\pm4\Re p_{n-m}$, $m=\overline{0,n-1}$:
	\begin{equation}\label{asqsol}
	q(x,t)=\left\{
	\begin{aligned}
	& \frac{2ip_{n-m}^2c_0^{as}(-\Re p_{n-m})}
	{p_{n-m}^2+c_0^{as}(-\Re p_{n-m})f_{n-m}^{as}(x_0,t)}+o(1),&
	t\to\infty,\quad x=-4\Re p_{n-m}t+x_0,\\
	&\frac{-2i\bar{p}^2_{n-m}\overline{f^{as}_{n-m}}(x_0,t)}
	{\overline{p}_{n-m}^2+\overline{c_0^{as}}(-\Re p_{n-m})
		\overline{f_{n-m}^{as}}(x_0,t)}+o(1),& t\to\infty,\quad x=4\Re p_{n-m}t-x_0,\\
	\end{aligned}
	\right.
	\end{equation}
	where $x_0\in\mathbb{R}$, $c_0^{as}(\xi)$ is given by (\ref{c0as}), and 
	$f_{n-m}^{as}(x_0, t)$ is given by
	\begin{equation}\label{fas}
	f_{n-m}^{as}(x_0, t)=\frac{\eta_{n-m}
		\exp\{2ip_{n-m}x_0-4it(\Re^2p_{n-m}+\Im^2p_{n-m})\}}
	{\dot a_1(p_{n-m}) \delta^2(p_{n-m},-\Re p_{n-m})},\quad m=\overline{0,n-1}.
	\end{equation}
	The asymptotics (\ref{asqsol}) is valid for all $t\gg0$ and $x_0\in\mathbb{R}$ such as
	$$
	p_{n-m}^2+c_0^{as}(-\Re p_{n-m})f_{n-m}^{as}(x_0, t) \neq 0.
	$$
	Moreover, as $x_0\to\pm\infty$, the asymptotics (\ref{asqsol}) match the
	asymptotics in the neighboring sectors in (\ref{asq1}).
\end{proposition}
\begin{proof}
	Similarly to item (i) in  Proposition \ref{asRH}, it can be shown that along the rays $\xi=-\Re p_{n-m}$, $m=\overline{0,n-1}$ (see Figure \ref{F1}), 
	the long-time behavior of  $q(x,t)$ can be described in terns of the solutions
	of the RH problem (cf. (\ref{RHas1}))
		\begin{subequations}\label{RHas1s}
		\begin{align}
		&M^{as}_+(\xi,t,k)=M^{as}_-(\xi,t,k)J^{as}(\xi,t,k),&& k\in\hat\Gamma,\\
		&M^{as}(\xi,t,k)\to I, && k\to\infty,\\
		\label{RHas1cs}
		&\underset{k=0}{\operatorname{Res}}M^{as\,(2)}(\xi,t,k)=
		c_{0}^{as}(\xi)M^{as\,(1)}(\xi,t,0),\\
		\label{singxi1s}
		&\underset{k=p_{n-m}}{\operatorname{Res}}M^{as\,(1)}(\xi,t,k)=
		f_{n-m}^{as}(x_0, t)M^{as\,(2)}(\xi,t,p_{n-m})\\
		&M^{as}(\xi,t,k)=
		\left(M^{as}_{-\xi}(\xi,t)+O(k+\xi)\right)
		(k+\xi)^{(\Im\nu(-\xi)-m)\sigma_3},&&
		k\to-\xi,
		\end{align}
	\end{subequations}
	where $J^{as}(\xi,t,k)$ is given by (\ref{RHas1j}) and $x_0\in\mathbb{R}$ parametrizes constant parallel shifts of the considered ray: 
	$x=-4\Re p_{n-m}t-x_0$, $m=\overline{0,n-1}$ (notice that such shift does not change the value of the slow variable $\xi=\frac{x}{4t}$ as $t\to\infty$).
	Using the Blaschke-Potapov factors (see e.g. \cite{FT, RSs}), the asymptotics of  $q(x,t)$ can be found in terms of the solution of a  regular Riemann-Hilbert problem:
	\begin{subequations}
		\label{sol+0}
		\begin{align}
		\label{sol+0a}
		q(x,t)&=2ip_{n-m}P_{12}(\xi,t)+
		2i\lim\limits_{k\to\infty}kM^{as,\,R}_{12}(\xi,t,k),
		\quad x>0,\\
		q(x,t)&=-2i\bar{p}_{n-m}\overline{P_{21}(-\xi,t)}
		-2i\lim\limits_{k\to\infty}
		k\overline{M^{as,\,R}_{21}(-\xi,t,k)},\quad x<0,
		\end{align}
	\end{subequations}
	where $M^{as,\,R}(\xi,t,k)$ solves the RH problem 
	\begin{subequations}\label{RHR}
		\begin{align}
		&M^{as,\,R}_+(\xi,t,k)=M^{as,\,R}_-(\xi,t,k)J^{as,\,R}(\xi,t,k),& k\in\hat\Gamma,\\
		&M^{as,\,R}(\xi,t,k)\rightarrow I, & k\rightarrow\infty,\\
		&M^{as,\,R}(\xi,t,k)=\left(M^{as,\,R}_{-\xi}(\xi,t)+O(k+\xi)\right)
		(k+\xi)^{(\Im\nu(-\xi)-m)\sigma_3}, & k\to-\xi,
		\end{align}
		with $\xi=-\Re p_{n-m}$, $m=\overline{0,n-1}$, and
		\begin{equation}
		\label{J^R}
		J^{as,\,R}(\xi,t,k)=
		\begin{pmatrix}
		1& 0\\
		0& \frac{k-p_{n-m}}{k}
		\end{pmatrix}
		J^{as}(\xi,t,k) 
		\begin{pmatrix}
		1& 0\\
		0& \frac{k}{k-p_{n-m}}
		\end{pmatrix},\quad k\in\hat{\Gamma}.
		\end{equation}
	\end{subequations}
	Here $P_{12}(\xi,t)$, $P_{21}(\xi,t)$ are determined in terms of $M^{as,\,R}(\xi,t,k)$ as follows:
	\begin{equation}
	\label{P}
	P_{12}(\xi,t)=\frac{g_1(\xi,t)h_1(\xi,t)}
	{g_1(\xi,t)h_2(\xi,t)-g_2(\xi,t)h_1(\xi,t)},\,
	P_{21}(\xi,t)=-\frac{g_2(\xi,t)h_2(\xi,t)}
	{g_1(\xi,t)h_2(\xi,t)-g_2(\xi,t)h_1(\xi,t)},
	\end{equation}
	where $g(\xi,t)=\left(
	\begin{smallmatrix}g_1(\xi,t)\\g_2(\xi,t) \end{smallmatrix}\right)$
	and  $h(\xi,t)=\left(
	\begin{smallmatrix}h_1(\xi,t)\\h_2(\xi,t) \end{smallmatrix}\right)$
	are given by
	\begin{subequations}\label{gh}
		\begin{align}
		g(\xi,t)&=p_{n-m}M^{as,\,R(1)}(\xi,t,p_{n-m})-
		f_{n-m}^{as}(x_0,t)M^{as,\,R(2)}(\xi,t,p_{n-m}),\\
		h(\xi,t)&=p_{n-m}M^{as,\,R(2)}(\xi,t,0)+ c_0^{as}(\xi)M^{as,\,R(1)}(\xi,t,0).
		\end{align}
	\end{subequations}
Setting $M^{as,R}(\xi,t,k)\approx I$ (as $t\to\infty$), one can calculate $g(\xi,t)$ and $h(\xi,t)$;  substituting them into (\ref{P}) and (\ref{sol+0}), 
the formulas for the main terms in (\ref{asqsol}) follow.
\end{proof}

Descriptions of  transition zones of type (2) and (3) are  open, challenging
 problems, which are 
 beyond the scope of this paper. 
For transition zones of type (2), we face the problem of winding of arguments of certain
spectral functions, see (\ref{windingRH}), whereas in the case when $\xi$  approaches
0 (the transition zone of type (3)), 
 we encounter another difficulty: the slow variable $\xi$ and the singularity of the Riemann-Hilbert problem (see (\ref{z}) and (\ref{zc}))  merge. For the focusing NNLS equation, we partially address the latter problem in \cite{RScv}, for the pure step initial data $q_0(x)=q_{0,A}(x)$, where we present a family of different asymptotic zones; particularly, the decaying zones (for $x<0$) are characterized by the 
decay of order $t^p\sqrt{\ln t}$, where $p<0$ parametrizes the family.
Applying similar ideas for the defocusing problem seems possible, 
but it will require substantial modifications since the behavior of the spectral functions as $k\to 0$ in the focusing and defocusing cases is different
(see item 5 of Proposition \ref{properties}).

\bigskip

\textbf{Concluding remark.}
\textit{
In the present work we have considered the Cauchy problem (\ref{1}), (\ref{2}) with the initial data close to the pure step function (\ref{shifted-step})  ``shifted to the right'', i.e., with $R>0$ (concerning the case $R=0$ see Remark \ref{rem4}). 
In the case $R<0$
(i.e., for initial data ``shifted to the left''), the spectral functions $a_j(k)$, $j=1,2$ and $b(k)$ associated with the pure step initial data (\ref{shifted-step})
can be explicitly calculated as well, but their analytical properties 
differ significantly from those in the case with $R>0$,  complicating the 
analysis of location of zeros (particularly, in this case $a_2(k)$ is not a constant)
and the respective winding properties of the spectral functions. 
The application of the IST method and subsequent asymptotic analysis in the case $R<0$ remain an open problem.
}


\end{document}